\pdfoutput=1
\documentclass[11pt]{article}

\usepackage{mymacros}

\begin{document}

\title{Étale structures and the Joyal--Tierney representation theorem in countable model theory}
\author{Ruiyuan Chen}
\date{}
\maketitle

\begin{abstract}
An étale structure over a topological space $X$ is a continuous family of structures (in some first-order language) indexed over $X$.
We give an exposition of this fundamental concept from sheaf theory and its relevance to countable model theory and invariant descriptive set theory.
We show that many classical aspects of spaces of countable models can be naturally framed and generalized in the context of étale structures, including the Lopez-Escobar theorem on invariant Borel sets, an omitting types theorem, and various characterizations of Scott rank.
We also present and prove the countable version of the Joyal--Tierney representation theorem, which states that the isomorphism groupoid of an étale structure determines its theory up to bi-interpretability; and we explain how special cases of this theorem recover several recent results in the literature on groupoids of models and functors between them.
\let\thefootnote=\relax
\end{abstract}

\tableofcontents

\section{Introduction}
\label{sec:intro}

A standard technique in countable model theory associates, to each (possibly infinitary) first-order theory $\@T$ in some language $\@L$, a topological space $X$ parametrizing all of its countable models up to isomorphism.
This enables the application to model theory of powerful tools from topology, dynamics, descriptive set theory, and computability theory, such as Baire category techniques and Borel complexity theory.
See \cite[\S16.C]{Kcdst}, \cite{BKgrp}, \cite[\S3.6, \S11.4]{Gidst}, \cite{Mcpstr}.

In fact, there are several variants of such a ``topological space of models'' $X$.
Perhaps the best-known consists of all models on a \emph{fixed} countably infinite set such as $\#N$, regarded as a subspace of a Cantor space $X \subseteq \prod_i 2^{\#N^{n_i}}$ specifying the interpretations of each relation and function symbol in the language $\@L$; see \cite[\S3.6]{Gidst}.
By taking instead models on (certain) subsets of $\#N$, one can also encompass finite models up to isomorphism; see \cite{Cscc}.
One can also broaden (e.g., to all finitary first-order) or narrow (e.g., to only atomic) the class of formulas which define an open subset of $X$; see \cite[\S11.4]{Gidst}, \cite{BFRSVle}.
Alternatively, one can consider \emph{marked models} over a fixed generating set, e.g., groups generated by $\#N$, hence quotients of the free group $\ang{\#N}$ by a normal subgroup, represented as the space of all normal subgroups of $\ang{\#N}$; see e.g., \cite{Tqiso}.

This landscape is clarified by the observation that such $X$, as a space of \emph{codes} of countable structures, may be meaningfully distinguished from the structures themselves.
That is, rather than a ``space of models'' $X$, one really has a topological space $X$ together with a ``continuous map''
\begin{align*}
X &--> \{\text{all $\@L$-structures}\} \\
x &|--> \@M_x
\end{align*}
where the right-hand side is \emph{not} a topological space, but a higher-order analog thereof, crucially differing in that two ``points'' (i.e., structures) may be ``equal'' (i.e., isomorphic) in more than one way.
Such a generalized ``space'' is made precise by the concept of the \emph{classifying topos} of $\@L$-structures.
Classifying toposes have been widely influential in such areas as algebraic geometry, topology, and category theory; see \cite{AGVsga4}, \cite{MMtopos}, \cite{Jeleph}.
However, to our knowledge, topos-theoretic ideas are not very well-known or used in the countable model theory literature, perhaps due to the substantial amount of category theory needed to define and work with them.%
\footnote{There are of course exceptions, such as \cite{Catomic}, \cite{DLwap}, \cite{Kubis}, as well as parts of model theory further from the countable realm where categorical tools are more routinely used, e.g., \cite{MRfocl}, \cite{BRpure}, \cite{Prest}, \cite{LRVfork}.}

\subsection{Étale structures}

The goal of this article is to give a self-contained development of the above perspective from a classical model-theoretic and descriptive set-theoretic angle, minimizing the category theory needed.
The central concept of a ``continuous family of structures'' $(\@M_x)_{x \in X}$ parametrized over a topological space $X$, as above, may be represented more concretely as an \defn{étale bundle of structures} $\@M -> X$, where $\@M$ is the disjoint union of the structures $\@M_x$ equipped with a global topology that captures the ``continuity'' of the fibers $\@M_x$ over $x \in X$.
This concept is well-known in topos theory, perhaps more commonly in the equivalent form of a \emph{sheaf of structures}.
However, by taking an étale structure $\@M -> X$ as the formal meaning of a ``continuous map'' $X -> \{\text{$\@L$-structures}\}$ as above, one may develop much of classifying topos theory (for countable structures and theories) in a point-set topological manner, without reference to toposes, sheaves, or other categorical notions.

A recurring theme of our account is that many standard concepts and constructions in point-set topology and descriptive set theory have \emph{precise} analogs for étale structures, thought of as ``continuous families of structures'' as above.
Often, these analogies yield more precise formulations of well-known folklore connections.
For instance, it is well-known that the classical Baire category theorem is closely related to the omitting types theorem in (infinitary) first-order logic, and in fact these two results may be used to prove each other; see e.g., \cite{ETomit}.
In \cref{sec:omittype}, we explain how the omitting types theorem is the literal first-order generalization of the Baire category theorem; and we show how the former may be reduced to the latter via an étale structure, as a higher-order instance of the fact that a continuous open surjection (is category-preserving, hence) may be used to transfer Baire category from its domain to its codomain.
\Cref{tbl:analogy} depicts various other topological concepts with an étale model-theoretic analog that we will discuss.

\begin{table}[hbtp]
\centering
\begin{tabularx}{\textwidth}{XX}
\toprule
topology &
étale model theory
\\
\midrule
continuous map $X --->{f} Y$ &
étale structure $\@M -> X$ (\cref{def:estr})
\\
continuous open map (onto its image) &
étale structure with $\Sigma_1$ saturations (\ref{def:estr-osat})
\\
Borel ($\*\Sigma^0_\alpha$) set $V \subseteq Y$ &
$\@L_{\omega_1\omega}$ ($\Sigma_\alpha$) formula $\phi$ (\ref{def:sigmapi})
\\
preimage $f^{-1}(V)$ of $V \subseteq Y$ &
interpretation $\phi^\@M$ of formula $\phi$ (\ref{rmk:estr-cts})
\\
image $f(U)$ of $U \subseteq X$ &
theory/type of $U \subseteq M$
\\
saturation $[U]_{\ker(f)} = f^{-1}(f(U))$ of $U \subseteq X$ &
saturation $\Iso_X(\@M) \cdot U$ of $U \subseteq M$
\\
Baire quantifier $\exists^*_f(U)$ of $U \subseteq X$ &
Vaught transform $U^{\triangle \Iso_X(\@M)}$ of $U \subseteq M$ (\ref{def:vaught})
\\
kernel $\ker(f) \subseteq X^2$ &
isomorphism groupoid $\Iso_X(\@M) -> X^2$ (\ref{def:isogpd})
\\
composition \smash{$Z -> X --->{f} Y$} &
pullback $Z \times_X \@M$ (\ref{rmk:estr-pb})
\\
change of topology on $Y$ &
Morleyization (\ref{def:estr-morley})
\\
\bottomrule
\end{tabularx}
\caption{Correspondence between topological and étale model-theoretic notions, when $Y$ is replaced with the ``space of all structures''.}
\label{tbl:analogy}
\end{table}

In the course of developing this dictionary, we also prove generalizations to arbitrary étale parametrizations of classical descriptive set-theoretic results known for \emph{specific} parametrizations.
For instance, the Lopez-Escobar theorem \cite{Llo1o} and its strengthening by Vaught \cite{Vaught} show that every Borel isomorphism-invariant set of models, in the classical space of countably infinite models on $\#N$, is axiomatizable by an infinitary formula, of the same quantifier complexity as the Borel complexity of the given set.
Recently in \cite{BFRSVle}, the authors prove a similar result, for a different parametrizing space, that takes positive (i.e., negation-free) formulas into account.%
\footnote{\label{ft:effective}%
The papers \cite{BFRSVle}, \cite{CMRpos}, \cite{HMMMcomp}, \cite{HMMborel} are also largely concerned with tracking the effective (lightface) complexity of formulas and models, which we do not consider at all in this paper.}
In \cite{Cscc}, we proved such a result that also takes finite models into account.
One could also ask whether a Lopez-Escobar theorem holds for the space of marked structures.
We prove a result encompassing all of these as special cases:

\begin{theorem}[Lopez-Escobar for étale structures; see \cref{thm:lopez-escobar}]
For any étale parametrization of countable structures $\@M -> X$ with $\Sigma_1$ saturations, every isomorphism-invariant $\*\Sigma^0_\alpha$ set of tuples $A \subseteq M^n_X$ is defined by a $\Sigma_\alpha$ formula.
\end{theorem}

Here $\Sigma_\alpha$ refers to the $\alpha$th level of the complexity hierarchy of the countably infinitary logic $\@L_{\omega_1\omega}$.
To say that the étale structure $\@M -> X$ has \defn{$\Sigma_1$ saturations} means that the isomorphism saturation of each open set in $M$ (or a finite power) is $\Sigma_1$-definable; see \cref{def:estr-osat}.
This is the key property on an étale structure that implies it is a ``nice'' parametrization of countable models obeying the usual results, and holds for all of the aforementioned ``spaces of countable models'' commonly considered in practice (subject to minor conditions on the theory); see \cref{sec:param}.

As shown in \cref{tbl:analogy}, this condition is analogous to a continuous map between topological spaces $f : X -> Y$ being open, or more precisely, being \emph{open onto its image} $f(X) \subseteq Y$, since that is equivalent to each open $U \subseteq X$ having open saturation by the equivalence relation $\ker(f) \subseteq X^2$ induced by $f$.
Such maps play an important role in descriptive set theory: for nice (e.g., Polish) spaces $X, Y$, the image $f(X) \subseteq Y$ must be $\*\Pi^0_2$, by a combination of classical results of Sierpinski \cite[8.19]{Kcdst} (that an open $T_3$ quotient of a Polish space is Polish) and Alexandrov \cite[3.11]{Kcdst} (that Polish subspaces are $\*\Pi^0_2$).
We prove the analogous result for étale structures:

\begin{theorem}[see \cref{thm:estr-osat-pi02}]
\label{intro:thm:estr-osat-pi02}
For any étale parametrization of countable structures $\@M -> X$ with $\Sigma_1$ saturations, over a (quasi-)Polish base $X$, the class of parametrized models (i.e., those isomorphic to a fiber of $\@M$) is axiomatizable among countable structures by a $\Pi_2$ sentence.
\end{theorem}

Here \emph{quasi-Polish} refers to a generalization \cite{dBqpol} of Polish spaces that we find more convenient; one definition is that they are precisely the open $T_0$ quotients of Polish spaces.
See \cref{def:qpol}.

In fact, \cref{intro:thm:estr-osat-pi02} generalizes not only the aforementioned classical results of Sierpinski and Alexandrov in topology, but also (by taking $X = 1$) the classical model-theoretic fact that an atomic model has a $\Pi_2$ Scott sentence (relative to a countable fragment of $\@L_{\omega_1\omega}$; see \cite[11.5.7]{Gidst}).
That these three classical results are connected at all may come as a surprise, and is an example of the conceptual clarification we believe is provided by the perspective of étale structures.

\subsection{The Joyal--Tierney and Moerdijk representation theorems}

Classically, given a continuous open surjection $f : X ->> Y$ between topological spaces, the topology on $Y$ is necessarily the quotient topology induced by the topology on $X$.
This fact means that if one is interested in studying a class of spaces $Y$, which can all be parametrized as open quotients of a restricted class of spaces $X$, then in some sense it ``suffices'' to study the spaces $X$ together with the equivalence relations $\ker(f) \subseteq X^2$.
For example, this provides one formal explanation of the idea that to do descriptive set theory, one really only needs to consider the Baire space $\#N^\#N$, whose open quotients yield all Polish spaces.

While this topological fact is rather easy, the analogous fact for étale structures with $\Sigma_1$ saturations, i.e., ``continuous open maps $X -> \{\text{all structures}\}$'', is much deeper, since the ``space'' we are parametrizing is a much more complicated object than the genuine space of parameters $X$.
The analog of the equivalence relation $\ker(f) \subseteq X^2$, which consists of pairs of points $x, y \in X$ such that $f(x) = f(y)$, is the \defn{isomorphism groupoid} $\Iso_X(\@M)$ of an étale structure $\@M -> X$, consisting of $x, y \in X$ \emph{together with an isomorphism} $g : \@M_x \cong \@M_y$.
There is a natural ``pointwise convergence'' topology on $\Iso_X(\@M)$ (generalizing the pointwise convergence topology on $\Aut(\@M)$ for a single countable structure $\@M$) that turns it into a topological groupoid with space of objects $X$; see \cref{def:isogpd}.
We now have one of the central results of topos theory, which we state and prove for the restricted context of countable structures and $\@L_{\omega_1\omega}$ in \cref{sec:jt}:

\begin{theorem}[Joyal--Tierney \cite{JTloc}; see \cref{thm:jt} and \cref{rmk:jt}]
\label{intro:thm:jt}
Let $\@M -> X$ be an étale parametrization of countable structures with $\Sigma_1$ saturations, and let $\@T$ be the $\Pi_2$ theory of its fibers from \cref{intro:thm:estr-osat-pi02}.
Then the $\Sigma_1$ imaginaries over $\@T$ are canonically equivalent (as a category) to the continuous étale actions of the groupoid $\Iso_X(\@M)$.
\end{theorem}

Here by \defn{$\Sigma_1$ imaginaries}, we mean the usual model-theoretic concept \cite[Ch.~5]{Hmod} adapted to the infinitary setting: namely, a $\Sigma_1$-definable quotient of a countable disjoint union of $\Sigma_1$-definable sets (represented syntactically as families of $\Sigma_1$ formulas); see \cref{def:imag}.
Analogously to the topological setting, this result really says that étale structures in some sense ``fully parametrize'' countable models, since the theory $\@T$ of the parametrized models may be ``recovered'' from the topological groupoid $\Iso_X(\@M)$.
More precisely, we recover $\@T$ up to $\Sigma_1$ bi-interpretability, since a theory is determined up to bi-interpretability by its imaginaries.

While \cref{intro:thm:jt} concerns the topological context, meaning on the model-theoretic side that we restrict to $\Sigma_1$ formulas, there is an analogous result in the Borel context, which shows that a theory may be recovered up to $\@L_{\omega_1\omega}$ bi-interpretability from its \emph{Borel} groupoid of isomorphisms.
The proof combines \cref{intro:thm:jt} with a (Becker--Kechris-type) topological realization theorem for groupoid actions from \cite{Cbk} in order to translate between the topological and Borel settings.
(As explained in \cref{sec:intro-connections} below, this generalizes results from \cite{Cscc}, \cite{HMMborel}.)

\begin{theorem}[see \cref{thm:scc} and \cref{rmk:scc}]
\label{intro:thm:scc}
Let $\@M -> X$ be an étale parametrization of countable structures with $\Sigma_1$ saturations, and let $\@T$ be the theory of its fibers.
Then the $\@L_{\omega_1\omega}$ imaginaries over $\@T$ are canonically equivalent (as a category) to the fiberwise countable Borel actions of the groupoid $\Iso_X(\@M)$.
\end{theorem}

As we explain in \cref{sec:interp}, it is in fact possible to make this recovery of a theory from its groupoid of models fully precise, provided that one is willing to make full use of category-theoretic language (which is why we postpone this discussion until \cref{sec:interp}).
We are saying that
\begin{align*}
\{\text{theories}\} &--> \{\text{groupoids}\} \\
\@T &|--> \text{isomorphism groupoid of some parametrization of $\@T$}
\end{align*}
is an ``embedding'': we can recover a theory $\@T$ from its isomorphism groupoid.
To make this precise, we should specify the kind of structure on the domain and codomain of this map.
Namely, they are \emph{2-categories}: between theories $\@T_1, \@T_2$, we have interpretations $\@F : \@T_1 -> \@T_2$ as morphisms; and between parallel interpretations, we have \emph{definable isomorphisms as 2-cells} (all either $\Sigma_1$ or $\@L_{\omega_1\omega}$):
\begin{equation*}
\begin{tikzcd}
\@T_1 \rar[bend left,"\@F"] \rar[bend right,"\@G"'] \rar[phantom,"\scriptstyle\Down h"] &
\@T_2 \rar["\@H"] &
\@T_3
\end{tikzcd}
\end{equation*}
Similarly, between groupoids, we have functors, between which we have natural isomorphisms (continuous or Borel, respectively).
See \cref{def:interp} and \cref{rmk:2cat} for details.

\begin{theorem}[see \cref{thm:interp2-equiv,thm:interp2-borel-equiv}]
\label{intro:thm:interp2-equiv}
We have contravariant equivalences of 2-categories
\begin{align*}
\brace*{\begin{gathered}
\text{countable $\Pi_2$ theories} \\
\text{$\Sigma_1$ interpretations} \\
\text{$\Sigma_1$-definable isomorphisms}
\end{gathered}}
&\simeq
\brace*{\begin{gathered}
\text{open non-Archimedean quasi-Polish groupoids} \\
\text{continuous functors} \\
\text{continuous natural transformations}
\end{gathered}},
\\
\brace*{\begin{gathered}
\text{countable $\@L_{\omega_1\omega}$ theories} \\
\text{$\@L_{\omega_1\omega}$ interpretations} \\
\text{$\@L_{\omega_1\omega}$-definable isomorphisms}
\end{gathered}}
&\simeq
\brace*{\begin{gathered}
\text{open non-Archimedean quasi-Polish groupoids} \\
\text{Borel functors} \\
\text{Borel natural transformations}
\end{gathered}},
\end{align*}
taking a theory to the isomorphism groupoid of some parametrization of it.
\end{theorem}

As alluded to above, the ``injectivity'' of these maps (more precisely, that they are fully faithful 2-functors) boils down to the Joyal--Tierney \cref{intro:thm:jt} and its Borel analog, \cref{intro:thm:scc}.

The ``surjectivity'', that every topological groupoid satisfying certain conditions listed on the right-hand side above (see \cref{def:gpd-subgpd} for details) can be represented as an isomorphism groupoid, boils down to a topos-theoretic argument of Moerdijk \cite{Mtop2}, which constructs from an \emph{abstract} topological groupoid $G$ (satisfying certain conditions) an étale structure $\@M$ over the objects of $G$.
We present a self-contained proof in \cref{thm:onagpd-dense}, and show that in the countable case, via a Baire category argument, we get a full representation of the groupoid $G$ as $\Iso(\@M)$; see \cref{thm:onagpd-qpol}.
This result simultaneously generalizes the Yoneda lemma from category theory (which shows that every \emph{discrete} groupoid is a canonically an isomorphism groupoid, via the left action on itself), and the classical result that non-Archimedean Polish groups are precisely the automorphism groups of countable structures (see e.g., \cite[2.4.4]{Gidst}).

\subsection{Connections to other work}
\label{sec:intro-connections}

This article was largely motivated by several recent works in the countable model theory literature showing connections between the syntax and semantics of infinitary logic.
The earliest inspiration for this line of work seems to be \cite{AZqfin}, showing that the automorphism group of an $\aleph_0$-categorical (in the usual finitary first-order sense) structure determines it up to (finitary) bi-interpretability.
Analogously, \cite{HMMborel} showed that the automorphism group of an arbitrary countable structure determines it up to $\@L_{\omega_1\omega}$ bi-interpretability.
In that paper, respectively \cite{HMMMcomp}, the authors also showed that Borel, respectively continuous, functors between the groupoids of all isomorphic copies of two structures on underlying set $\#N$ correspond to $\@L_{\omega_1\omega}$, respectively $\Sigma_1$, interpretations between those structures.
In \cite{CMRpos}, the authors prove a finer result in the continuous context keeping more detailed track of negations in formulas.

From the étale perspective, these results can be seen as special cases of the Joyal--Tierney \cref{intro:thm:jt} and its Borel version, \cref{intro:thm:scc}, for specific examples of étale parametrizations $\@M -> X$ (whose fibers are all isomorphic).\cref{ft:effective}
The special case of bi-interpretations follows from general categorical considerations, once Joyal--Tierney is recast into its 2-categorical ``injectivity'' form of \cref{intro:thm:interp2-equiv}.
In \cite{Cscc}, whose main result was yet another special case of \cref{intro:thm:scc} (again for a specific parametrization), we made this categorical perspective explicit, including the connections to the Joyal--Tierney and Becker--Kechris theorems (as noted above \cref{intro:thm:scc}).

We should also note that, in the categorical literature, there have been previous point-set topological presentations of the Joyal--Tierney theorem, such as \cite{BMgpd} and \cite{AFscc}.
The latter paper also used the Joyal--Tierney machinery to give a 2-categorical ``duality'' between theories and groupoids, as in \cref{intro:thm:interp2-equiv}, but for finitary first-order logic, and again considering only a specific étale parametrization (or rather, its sheaf-theoretic formulation).

We view this paper as a sequel to \cite{Cscc}, that is informed by the realization that working with an abstract étale structure, rather than a specific ``space of countable models'' as was done in that paper, yields a cleaner and more general presentation.
In the preprint \cite{Cgpd}, we proved the Borel version of \cref{intro:thm:interp2-equiv} using \cite{Cscc} and the aforementioned Moerdijk groupoid representation, in the process developing the needed parts of countable étale model theory in a rather dense fashion; the present paper is a more comprehensive and (hopefully) readable rewrite of the first half of \cite{Cgpd}.
The second half of \cite{Cgpd}, dealing with continuous logic for metric structures \cite{BBHU}, will be subsumed in a future work that also includes a metric version of the Joyal--Tierney theorem and \cite{Cscc}.

\paragraph*{Acknowledgments}

I would like to thank Anush Tserunyan for providing useful comments on a draft of this paper.
Research supported by NSF grant DMS-2224709.

\section{Preliminaries on topology}

Recall that a \defn{Polish space} is a second-countable, completely metrizable topological space.
We will be working in a more general non-Hausdorff setting, that obeys better closure properties:

\begin{definition}[{Selivanov \cite{Sdst}}]
\label{def:borel}
In an arbitrary topological space $X$, possibly non-metrizable (hence closed sets may not be $G_\delta$), we define the \defn{Borel hierarchy} as follows:
\begin{itemize}
\item
A $\*\Sigma^0_1$ set is an open set.
\item
For $\alpha \ge 2$, a $\*\Sigma^0_\alpha$ set is a countable union
$
\bigcup_i (A_i \setminus B_i)
$
where each $A_i, B_i$ is $\*\Sigma^0_\beta$ for some $\beta < \alpha$.
\item
A $\*\Pi^0_\alpha$ set is the complement of a $\*\Sigma^0_\alpha$ set, thus for $\alpha \ge 2$ is of the form
$\bigcap_i (A_i => B_i)$,
where $A_i, B_i \in \*\Sigma^0_\beta$ for some $\beta < \alpha$,
and $A_i => B_i := \neg A_i \cup B_i$.
As usual, $\*\Delta^0_\alpha$ means $\*\Sigma^0_\alpha$ and $\*\Pi^0_\alpha$.
\end{itemize}
\end{definition}

\begin{example}
In a second-countable $T_0$ space $X$, the equality relation is $\*\Pi^0_2$ in $X^2$:
\begin{align*}
x = y  \iff  \forall \text{ basic open } U\, (x \in U \iff y \in U).
\end{align*}
\end{example}

\begin{definition}
\defn{Sierpinski space} is the topological space $\#S := 2$ with $\{1\}$ open but not closed.
\end{definition}

\begin{definition}[{de~Brecht \cite{dBqpol}\footnotemark}]
\footnotetext{This was not the original definition used in \cite{dBqpol}, but was proved to be equivalent in that paper.}
\label{def:qpol}
A \defn{quasi-Polish space} is a topological space homeomorphic to a $\*\Pi^0_2$ subspace of $\#S^\#N$.
\end{definition}

In \cite{dBqpol}, de~Brecht introduced quasi-Polish spaces and proved that they satisfy almost all of the usual descriptive set-theoretic properties of Polish spaces, including the following which we will freely use.
For proofs, see \cite{dBqpol} or \cite{Cqpol}.
\begin{eqenum}

\item \label{it:qpol-t0-cb-baire}
All quasi-Polish spaces are $T_0$, second-countable, and obey the Baire category theorem (which automatically then holds also for intersections of dense $\*\Pi^0_2$ sets, not just dense $G_\delta$).

\item \label{it:qpol-pol}
A topological space is Polish iff it is quasi-Polish and regular ($T_3$).

\item \label{it:qpol-dis}
If $X$ is a quasi-Polish space, and $A_i \subseteq X$ are countably many $\*\Sigma^0_\alpha$ sets, then there is a finer quasi-Polish topology containing each $A_i$ and contained in $\*\Sigma^0_\alpha(X)$.
In more detail,
\begin{enumerate}[label=(\alph*)]
\item \label{it:qpol-dis:clopen}
adjoining a single $\*\Delta^0_2$ set to the topology of $X$ preserves quasi-Polishness;
\item \label{it:qpol-dis:join}
if the intersection of countably many quasi-Polish topologies contains a quasi-Polish topology, then their union generates a quasi-Polish topology.
\end{enumerate}

\item \label{it:qpol-sbor}
A quasi-Polish space can be made zero-dimensional Polish by adjoining countably many closed sets to the topology.

\item \label{it:qpol-prod}
Countable products of quasi-Polish spaces are quasi-Polish.

\item \label{it:qpol-loc}
A space with a countable cover by open quasi-Polish subspaces is quasi-Polish.
In particular, a countable disjoint union of quasi-Polish spaces is quasi-Polish.

\item \label{it:qpol-pi02}
A subspace of a quasi-Polish space is quasi-Polish iff it is $\*\Pi^0_2$.

\item \label{it:qpol-openquot}
A continuous open $T_0$ quotient of a quasi-Polish space is quasi-Polish.
Conversely, every quasi-Polish space is an open quotient of a zero-dimensional Polish space.

\end{eqenum}

\begin{definition}
\label{def:lowpow}
For a topological space $X$, its \defn{lower powerspace} $\@F(X)$ is the space of all closed subsets $F \subseteq X$ equipped with the topology generated by the open subbasis consisting of
\begin{align*}
\Dia U := \{F \in \@F(X) \mid F \cap U \ne \emptyset\}
\end{align*}
for each open $U \subseteq X$.
Equivalently, fix any open basis $\@U$ for $X$; then the topology on $\@F(X)$ is induced via the embedding
\begin{align*}
\yesnumber
\label{eq:lowpow}
\@F(X) &--> \#S^{\@U} \\
F &|--> \{U \in \@U \mid F \cap U \ne \emptyset\}.
\end{align*}
If $X$ is quasi-Polish, then so is $\@F(X)$ \cite{dBKpowsp} (see also \cref{eq:estr-osat-pi02:coidl} below).

We will also need the \defn{fiberwise lower powerspace} $\@F_X(Y)$ of a continuous map $p : Y -> X$ (thought of as a continuous bundle of spaces $Y_x := p^{-1}(x)$ indexed over $X$), which is the disjoint union of the lower powerspaces of the fibers $\@F(Y_x)$, i.e.,
\begin{equation*}
\@F_X(Y) := \{(x, F) \mid x \in X \AND F \in \@F(Y_x)\},
\end{equation*}
equipped with the topology generated by the first coordinate projection $\@F_X(Y) -> X$ as well as
\begin{equation*}
\Dia_X U := \{(x, F) \in \@F_X(Y) \mid F \cap U \ne \emptyset\}
\end{equation*}
for each open $U \subseteq Y$.
When $X, Y$ are quasi-Polish, so is $\@F_X(Y)$ \cite[2.5.9]{Cbk}.
\end{definition}

\section{Preliminaries on infinitary logic}
\label{sec:lo1o}

We now review some conventions on infinitary first-order logic.
Throughout, $\@L$ will denote a countable first-order language.
We note that we allow both nullary function symbols (i.e., constants) and nullary relation symbols (i.e., atomic propositions).
In particular, $\@L$ could consist entirely of nullary relation symbols, in which case we are working essentially with propositional logic.
Everything we do will also generalize straightforwardly to multi-sorted languages, although we will rarely bother to spell out the details of such a generalization.

The \defn{countably infinitary first-order logic $\@L_{\omega_1\omega}$} is constructed like the usual finitary first-order logic, but allowing countable conjunctions $\bigwedge_i \phi_i(\vec{x})$ and disjunctions $\bigvee_i \phi_i(\vec{x})$ of formulas.
We will use the symbols $\top, \bot$ for nullary conjunction and disjunction.
We adopt the convention that each $\@L_{\omega_1\omega}$ formula $\phi(\vec{x})$ may only have finitely many free variables $\vec{x}$.
For general background on $\@L_{\omega_1\omega}$, see \cite[Ch.~6]{AKstr}, \cite[Ch.~11]{Gidst}, \cite{Mlo1o}.

For a first-order $\@L$-structure $\@M$ and formula $\phi(x_0,\dotsc,x_{n-1})$, we write $\phi^\@M \subseteq M^n$ for the interpretation of $\phi$ in $\@M$.
We always allow structures to be empty.

There are completeness theorems for various deductive systems for $\@L_{\omega_1\omega}$ \cite{Klo1o}, \cite{Llo1o}, \cite{FGformal}, which say that an $\@L_{\omega_1\omega}$ sentence has a proof iff it is true in every countable $\@L$-structure.

\begin{definition}[complexity classes of $\@L_{\omega_1\omega}$ formulas]
\label{def:sigmapi}
\leavevmode
\begin{itemize}
\item
A \defn{basic formula} is a finite conjunction of atomic formulas (no negations, not even $\ne$).
\item
A \defn{positive-primitive formula} is one built from atomic formulas using $\wedge, \top, \exists$.
Up to logical equivalence, these may all be written in the prenex normal form $\exists \vec{y}\, \phi(\vec{x},\vec{y})$ where $\phi$ is basic; we will also call this a $\exists{\wedge}$ formula.
\item
A \defn{$\Sigma_1$ formula} is built from atomic formulas using $\wedge, \top, \exists, \bigvee$.
These have a ${\bigvee}\exists{\wedge}$ normal form.
\item
For $\alpha \ge 2$, a \defn{$\Sigma_\alpha$ formula} is built from $\Sigma_\beta$ formulas and their negations, for $\beta < \alpha$, using $\wedge, \top, \exists, \bigvee$.
These have a normal form
\begin{align*}
\bigvee_i \exists \vec{y}_i\, \paren[\big]{\phi_i(\vec{x},\vec{y}_i) \wedge \neg \psi_i(\vec{x},\vec{y}_i)}
\end{align*}
where each $\phi_i, \psi_i$ is $\Sigma_\beta$ for some $\beta < \alpha$.
\item
For $\alpha \ge 2$, a \defn{$\Pi_\alpha$ formula} is the dual of a $\Sigma_\alpha$ formula, or equivalently, for $\phi_i,\psi_i$ as above,
\begin{align*}
\bigwedge_i \forall \vec{y}_i\, \paren[\big]{\phi_i(\vec{x},\vec{y}_i) -> \psi_i(\vec{x},\vec{y}_i)}.
\end{align*}
\end{itemize}
\end{definition}

\begin{convention}
\label{cvt:geom}
We will assume whenever convenient that every $\@L_{\omega_1\omega}$ formula is built using $\wedge, \top, \exists, \bigvee, \neg$ only, with $\bigwedge, \forall$ treated as abbreviations for $\neg \bigvee \neg, \neg \exists \neg$ respectively.
Thus, every formula is $\Sigma_\alpha$ in the above sense for some $\alpha$.
\end{convention}

The above definitions of $\Sigma_\alpha, \Pi_\alpha$ are not the standard ones found in e.g., \cite{AKstr}, which count all finitary quantifier-free formulas as $\Sigma_1$.
(They are \emph{almost} the ``positive $\Sigma^p_\alpha$, $\Pi^p_\alpha$ formulas'' of \cite{BFRSVle}, except that we do not even admit $\ne$ as $\Sigma_1$.)
While the standard definitions interact well with the Borel hierarchy of the usual zero-dimensional Polish space parametrizing $\@L$-structures on $\#N$ (see \cite[Ch.~11]{Gidst}, \cite[\S16.C]{Kcdst}, and \cref{ex:estr-infdec} below), the above definitions are the natural counterpart to Selivanov's non-Hausdorff Borel hierarchy (\cref{def:borel}) for quasi-Polish spaces.

Just as the non-Hausdorff Borel hierarchy generalizes the usual Borel hierarchy for Polish spaces, so too can the above ``non-Hausdorff'' hierarchy of formulas be regarded as generalizing the usual hierarchy.
For given a theory $\@T$, we may replace it with another $\@T'$ whose $\Sigma_\alpha, \Pi_\alpha$ formulas, in the above sense, correspond to the usual $\Sigma_\alpha, \Pi_\alpha$ formulas of $\@T'$, via the following standard device:

\begin{definition}
\label{def:morley}
A \defn{fragment} of $\@L_{\omega_1\omega}$ is a set $\@F$ of formulas which contains all atomic formulas and is closed under variable substitutions and subformulas.
(Note that this generalizes Sami's notion of fragment from \cite{Svaught} which requires closure under $\neg$, which in turn generalizes the more common notion requiring closure under all finitary operations \cite[11.2.3]{Gidst}.)

The \defn{Morleyization} of a fragment $\@F$ of $\@L_{\omega_1\omega}$ consists of an expanded language $\@L' \supseteq \@L$ with a new $n$-ary relation symbol $R_\phi$ for each $\phi(x_0,\dotsc,x_{n-1}) \in \@F$, together with the following $\Pi_2$ $\@L'$-theory $\@T'$ which says that this is an expansion by definitions:
\begin{align*}
\forall \vec{x}\, \paren[\big]{R_\phi(\vec{x}) &<-> \phi(\vec{x})} \quad \text{for atomic $\phi$}, \\
\forall \vec{x}\, \paren[\big]{R_{\phi \wedge \psi}(\vec{x}) &<-> R_\phi(\vec{x}) \wedge R_\psi(\vec{x})}, \\
\forall \vec{x}\, \paren[\big]{R_\top(\vec{x}) &<-> \top}, \\
\forall \vec{x}\, \paren[\big]{R_{\bigvee_i \phi}(\vec{x}) &<-> \bigvee_i R_\phi(\vec{x})}, \\
\forall \vec{x}\, \paren[\big]{R_{\exists y\, \phi}(\vec{x}) &<-> \exists y\, R_\phi(\vec{x},y)}, \\
\forall \vec{x}\, \paren[\big]{R_\phi(\vec{x}) \wedge R_{\neg \phi}(\vec{x}) &-> \bot}, \\
\forall \vec{x}\, \paren[\big]{\top &-> R_\phi(\vec{x}) \vee R_{\neg \phi}(\vec{x})},
\end{align*}
whenever the formulas in the subscripts are in $\@F$, and adopting here \cref{cvt:geom}.
Then each $\@L$-structure $\@M$ has a unique expansion to a model $\@M'$ of $\@T'$, where $R_\phi^{\@M'} := \phi^\@M$; and $\Sigma_\alpha$ $\@L'$-formulas $\phi'$ up to equivalence mod $\@T'$ are in bijection with $\@L$-formulas $\phi$ which are \defn{``$\Sigma_\alpha$ over $\@F$''}, meaning constructed as in $\Sigma_\alpha$-formulas but replacing atomic formulas with formulas in $\@F$.
We also call such an expansion $\@M'$ a \defn{Morleyization of the structure $\@M$} (over by the fragment $\@F$).

If we start with an $\@L$-theory $\@T$, then we may instead add the above axioms to $\@T$ to obtain a new theory $\@T' \supseteq \@T$, whose models correspond bijectively to models of $\@T$.
We may furthermore replace each original axiom $\phi \in \@T$ which is ``$\Pi_\alpha$ over $\@F$'' with the equivalent $\Pi_\alpha$ $\@L'$-axiom $\phi'$.
\end{definition}

\begin{example}
\label{ex:morley-neg}
If we Morleyize the fragment $\@F$ consisting of all atomic and negations of atomic formulas, we obtain a theory $\@T'$ modulo which the $\Sigma_\alpha, \Pi_\alpha$ $\@L'$-formulas, in our ``non-Hausdorff'' sense, are equivalently the $\Sigma_\alpha, \Pi_\alpha$ $\@L$-formulas in the traditional sense.
\end{example}

\begin{example}
\label{ex:prop}
Suppose $\@L$ is a countable \emph{propositional} language, consisting only of nullary relation symbols.
It is best to regard such a language as being $0$-sorted, so that an ``$\@L$-structure'' has no underlying set, and is simply a truth assignment $\@M : \@L -> 2$.
The $\Sigma_1$ formulas define the open sets of a topology, namely the Sierpinski power topology on $\#S^\@L$ (also known as the \emph{Scott topology}).
The $\Sigma_\alpha, \Pi_\alpha$ formulas define precisely the $\*\Sigma^0_\alpha, \*\Pi^0_\alpha$ subsets of $\#S^\@L$.
Thus, quasi-Polish spaces are up to homeomorphism precisely the spaces of models of countable $\Pi_2$ propositional theories.

The Morleyization of the preceding remark corresponds in this case to changing topology on $\#S^\#N$ to turn it into $2^\#N$, as in \cref{it:qpol-sbor}.
More generally, Morleyization of an arbitrary countable fragment $\@F$ corresponds to changing topology to make countably many Borel sets open as in \cref{it:qpol-dis}.

(In \cref{def:estr-morley}, we relate Morleyization of \emph{non-propositional} languages to change of topology.)
\end{example}

\begin{remark}
\label{rmk:types}
More generally, given an arbitrary countable language $\@L$ and countable $\Pi_2$ $\@L$-theory $\@T$, and $n \in \#N$, the $\Sigma_1$ formulas $\phi(x_0,\dotsc,x_{n-1})$ up to $\@T$-provable equivalence yield a quasi-Polish topology on the space of \defn{$\Sigma_1$ $n$-types of $\@T$}, meaning countably prime (i.e., complement closed under countable disjunctions) filters of such formulas, or equivalently (by the completeness theorem) complete $\Sigma_1$ theories of $n$-tuples in countable models of $\@T$.
When $n = 0$, this space has equivalent names in other contexts: the \emph{localic reflection} of a topos (see \cite[A4.6.12]{Jeleph}), and the \emph{topological ergodic decomposition} of a Polish group(oid) action (see \cite[10.3]{Kerg}, \cite[\S2]{Cmd}).
\end{remark}

\section{Étale spaces and structures}

\begin{definition}
Let $X$ be a topological space.
An \defn{étale space over $X$} is another topological space $Y$ equipped with a continuous map $p : Y -> X$ which is a local homeomorphism, i.e., $Y$ has an open cover consisting of open sets $S \subseteq Y$ such that $p|S : S -> X$ is an open embedding.
We will refer to such an $S \subseteq Y$ as an \defn{open section}, and to $p$ as the \defn{projection map}.
\end{definition}

We may think of an étale space as a bundle of discrete sets indexed over $X$, namely the fibers $Y_x := p^{-1}(x)$, with the discreteness witnessed uniformly across all fibers by the open sections.
For this reason, when discussing étale spaces, we usually emphasize the space $Y$ rather than the map $p$.
For general background on étale spaces (and their close relatives, sheaves), see \cite{Tsh}.

\begin{remark}
By \cref{it:qpol-loc}, for a quasi-Polish base space $X$, an étale space $Y$ over it is quasi-Polish iff it is second-countable, in which case it has a countable basis of open sections.
\end{remark}

\begin{definition}
\label{def:pb}
If $p : Y -> X$ is an étale space, and $f : Z -> X$ is a continuous map, we write $f^*(Y) = Z \times_X Y$ for the \defn{pullback} or \defn{fiber product}
\begin{equation*}
f^*(Y) = Z \times_X Y := \{(z,y) \in Z \times Y \mid f(z) = p(y)\}
\end{equation*}
fitting into a commutative square:
\begin{equation*}
\begin{tikzcd}
f^*(Y) \dar["f^*(p)"'] \rar & Y \dar["p"] \\
Z \rar["f"] & X
\end{tikzcd}
\end{equation*}
Here $f^*(p)$ is the first coordinate projection, and turns $f^*(Y)$ into an étale space, with a cover by open sections $f^*(S)$ for open sections $S \subseteq Y$.
The second projection yields canonical bijections
\begin{equation*}
f^*(Y)_z \cong Y_{f(z)}.
\end{equation*}

If $Z$ is also étale over $X$, then so is $f^*(Y) = Z \times_X Y$, with open sections consisting of $T \times_X S$ for open sections $T \subseteq Z$ and $S \subseteq Y$.
We also write
\begin{equation*}
Y^n_X := Y \times_X \dotsb \times_X Y
\end{equation*}
for the $n$-fold fiber power.
When $n = 0$, by convention $Y^0_X := X$.
Sometimes, it is convenient to use the following alternative (canonically isomorphic to the above) definition of $Y^n_X$ that explicitly records the basepoint in $X$, hence extends ``uniformly'' to the case $n = 0$:
\begin{equation}
\label{eq:pb-based}
Y^n_X := \set{(x,y_0,\dotsc,y_{n-1}) \in X \times Y^n}{x = p(y_0) = \dotsb = p(y_{n-1})}.
\end{equation}
\end{definition}

\begin{remark}
It is easily seen that an equivalent characterization of local homeomorphism $p : Y -> X$ is that both $p$ and the diagonal embedding $Y `-> Y \times_X Y$ are open maps.
The latter can be thought of as saying that the equality relation is open in each fiber, uniformly across all fibers.
\end{remark}

\begin{remark}
If $p : Y -> X$ and $q : Z -> X$ are two étale spaces over the same base, and $f : Y -> Z$ is a continuous map \defn{over $X$}, i.e., making the triangle
\begin{equation*}
\begin{tikzcd}
Y \drar["p"'] \ar[rr,"f"] && Z \dlar["q"] \\
& X
\end{tikzcd}
\end{equation*}
commute, then $f$ is automatically open, since for an open $p$-section $S \subseteq Y$, we have
\begin{align*}
f(S) = \bigcup_T \paren[\big]{T \cap q^{-1}(p(S \cap f^{-1}(T)))}
\end{align*}
where $T$ varies over open $q$-sections in $Z$.
It follows that each such $S$ is an open $f$-section, whence $f$ is in fact étale.
\end{remark}

\begin{remark}
If $p : Y -> X$ is a second-countable étale space, then $p$ (is not only open but) maps $\*\Sigma^0_\alpha$ sets onto $\*\Sigma^0_\alpha$ sets, since this is clearly true for subsets of each open section $S \subseteq Y$.

It follows that in the preceding remark, assuming $Y$ is second-countable, $f$ also maps $\*\Sigma^0_\alpha$ sets onto $\*\Sigma^0_\alpha$ sets.
\end{remark}

\begin{definition}
\label{def:estr}
Let $\@L$ be a countable language, $X$ be a topological space.
An \defn{étale $\@L$-structure} $\@M$ over $X$ is defined by ``internalizing the usual definition of first-order structure into the category of étale spaces over $X$''.
Concretely, $\@M$ consists of the following data:
\begin{itemize}
\item  an underlying étale space $p : M -> X$ (in place of an underlying set);
\item  for each $n$-ary function symbol $f \in \@L$, a continuous map $f^\@M : M^n_X -> M$ over $X$;
\item  for each $n$-ary relation symbol $R \in \@L$, an open set $R^\@M \subseteq M^n_X$.
\end{itemize}
By an abuse of notation, we will often refer to these data as ``the étale structure $p : \@M -> X$''.

Given such $\@M$, we may inductively interpret each term $\tau(x_0,\dotsc,x_{n-1})$ in the obvious manner to get a continuous map $\tau^\@M : M^n_X -> M$ over $X$.
We may then interpret each $\Sigma_\alpha$ $\@L_{\omega_1\omega}$ formula $\phi(x_0,\dotsc,x_{n-1})$ to get a subset $\phi^\@M \subseteq M^n_X$, which is $\*\Sigma^0_\alpha$ assuming that $M$ is second-countable, by an easy induction using the three preceding remarks in the $=, \exists$ cases.

If $\phi$ is a sentence, then $\phi^\@M \subseteq M^0_X = X$ (recall \cref{def:pb}); if $\phi^\@M = X$, then we call $\@M$ an \defn{étale model of $\phi$}, denoted $\@M |= \phi$.
\end{definition}

\begin{remark}
\label{rmk:estr-cts}
For an étale structure $p : \@M -> X$, for each $x \in X$, we may restrict everything to the fiber over $x$ to get an ordinary structure $\@M_x$.
Thus, we may think of $\@M$ as a ``continuous bundle of discrete structures, indexed over $X$'', or as a ``continuous map'' (as explained in the introduction):
\begin{align*}
\@M : X &--> \{\text{all $\@L$-structures}\} \\
x &|--> \@M_x.
\end{align*}
Here the right-hand side is meant in a vague, informal sense; we do not literally mean a continuous map to a topological space whose elements are $\@L$-structures.
The whole point is that $\@M$ provides a way to regard $X$ itself as a ``space of structures''.
(One may make the above map precise using the \emph{classifying topos of $\@L$-structures}; see \cite[D3]{Jeleph}.)

For each $n$-ary formula $\phi$, the fibers $\phi^\@M_x$ of $\phi^\@M \subseteq M^n_X$ are the ordinary interpretations $\phi^{\@M_x}$.
In particular, for a sentence $\phi$, to say that $\@M |= \phi$ just means $\@M_x |= \phi$ for each $x$.
\end{remark}

\begin{example}
\label{ex:estr-prop}
For a countable propositional language $\@L$, regarded as a 0-sorted first-order language as in \cref{ex:prop}, an étale $\@L$-structure $p : \@M -> X$ consists simply of an open set $P^\@M \subseteq X$ for each 0-ary relation symbol $P \in \@L$, hence literally corresponds to a continuous map $X -> \#S^\@L$ (whose $P$th coordinate is the indicator function of $P^\@M$).
\end{example}

\begin{remark}
\label{rmk:estr-pb}
If $p : \@M -> X$ is an étale structure, and $f : Z -> X$ is a continuous map, then we have a \defn{pullback structure} $f^*(\@M) -> Z$, given by pulling back the underlying étale space $M$ and the interpretations of all the symbols (\cref{def:pb}).
Clearly, for $z \in Z$ and a formula $\phi$,
\begin{align*}
f^*(\@M)_z &\cong \@M_{f(z)}, \\
\phi^{f^*(\@M)} &= f^*(\phi^\@M).
\end{align*}
In particular, if $\@M$ satisfies some theory, then so does $f^*(\@M)$.
\end{remark}

\begin{definition}
\label{def:estr-morley}
Let $p : \@M -> X$ be a second-countable étale structure, and $\@F$ be a countable fragment of $\@L_{\omega_1\omega}$ as in \cref{def:morley}, with Morleyization $\@T'$.
We may then Morleyize each fiber $\@M_x$ over $\@F$ as in \cref{def:morley}; however, the resulting bundle may no longer be an étale structure in the original topologies on $M, X$.
Instead, we refine the topologies as follows.

Let $X'$ be $X$ with the finer topology generated by all subbasic open sets of the form
\begin{align*}
p(U \cap \phi^\@M) = \set*{x \in X}{\exists \vec{a} \in U_x\, (\phi^{\@M_x}(\vec{a}))},
\end{align*}
where $U \subseteq M^n_X$ is a (basic) open set and $\phi \in \@F$ is an $n$-ary formula.
This topology indeed refines that of $X$, since we may take $n = 0$ and $\phi = \top$.
Since for $m$-ary $\phi$ and $n$-ary $\psi$,
\begin{align*}
p(U \cap \phi^\@M) \cap p(V \cap \psi^\@M) = p\paren*{(U \times_X V) \cap (\phi(x_0, \dotsc, x_{m-1}) \wedge \psi(x_m, \dotsc, x_{m+n-1}))^\@M},
\end{align*}
a basis for this topology consists of all sets $p(U \cap \phi^\@M)$ as above but now with $\phi$ ``basic over $\@F$'', i.e., a finite conjunction of formulas in $\@F$.

Now let $p' : \@M' -> X'$ be the pullback of $\@M$ along the identity $X' -> X$.
Thus, for each $n \in \#N$, an open basis for $M'^n_{X'}$ consists of
\begin{equation*}
U \cap p^{-1}(p(V \cap \phi^\@M))
\end{equation*}
for $U \subseteq M^n_X$ (basic) open, $V \subseteq M^m_X$ (basic) open for some other $m \in \#N$, and $\phi$ a finite conjunction of $m$-ary formulas in $\@F$.
For $U = V \subseteq M^n_X$ an open section, the above set becomes simply $U \cap \phi^\@M$; taking a union over all $U$ yields that $\phi^\@M \subseteq M'^n_{X'}$ is open for each $n$-ary $\phi \in \@F$, so that we may canonically expand $\@M'$ to an étale model of $\@T'$ over $X'$, which we call the \defn{Morleyization of the étale structure $\@M$} (over $\@F$).
\end{definition}

\begin{lemma}
If $X$ above is quasi-Polish, then so is $X'$.
\end{lemma}
\begin{proof}
Adjoining each $p(U \cap \phi^\@M)$ (in the notation above) to the topology of $X$ yields a quasi-Polish topology, by an induction on $\phi$ using \cref{it:qpol-dis}\cref{it:qpol-dis:clopen} and \cref{it:qpol-dis:join} to handle $\neg, \bigvee, \exists$.
\end{proof}

\begin{remark}
\label{rmk:estr-morley}
From the perspective of \cref{rmk:estr-cts}, the ``space of all models of $\@T'$'' can be seen as the ``space of all $\@L$-structures, with a finer topology (generated by $\@F$)''.
The Morleyized base space $X'$ above is then the result of ``pulling back this finer topology along $\@M : X -> \{\text{all $\@L$-structures}\}$'':
\begin{equation*}
\begin{tikzcd}
X' \dar["\@M'"'] \rar & X \dar["\@M"] \\
\{\text{all $\@T'$-models}\} \rar & \{\text{all $\@L$-structures}\}
\end{tikzcd}
\end{equation*}
\end{remark}

\begin{definition}
\label{def:isogpd}
Let $p : \@M -> X$ be an étale structure.
Its \defn{isomorphism groupoid} $\Iso_X(\@M)$ is the set of all triples $(x,y,g)$ where $x, y \in X$ and $g : \@M_x \cong \@M_y$ is an isomorphism.
We equip it with the topology generated by the first and second projections to $X$ as well as the subbasic open sets
\begin{align*}
\brbr{U |-> V} := \set[\big]{(x,y,g) \in \Iso_X(\@M)}{g(U_x) \cap V_y \ne \emptyset}
\end{align*}
for open $U, V \subseteq M$.
Equivalently, $\Iso_X(\@M)$ has an open basis consisting of the sets $\brbr{U |-> V}$, defined in the same way, for all open $U, V \subseteq M^n_X$ where $n \in \#N$.

Since the expression $\brbr{U |-> V}$ is easily seen to preserve unions in each of the variables $U, V$, it is enough to consider $U, V$ in any basis, e.g., a basis of open sections.
This reveals the topology on $\Iso_X(\@M)$ to be a generalization of the usual pointwise convergence topology on the automorphism group of a single structure, to which $\Iso_X(\@M)$ reduces when $X = 1$.

As its name suggests, the isomorphism groupoid $\Iso_X(\@M)$ is (the space of morphisms of) a \defn{topological groupoid} with space of objects $X$.
The domain and codomain of a morphism are
\begin{align*}
\dom(x,y,g) &:= x, &
\cod(x,y,g) &:= y.
\end{align*}
These maps $\dom, \cod : \Iso_X(\@M) -> X$ and the evident composition, identity, and inverse operations are easily seen to be continuous.
(For the abstract notion of topological groupoid, see \cref{def:gpd-subgpd}.)

We have a natural action of $\Iso_X(\@M)$ on the underlying étale space $M -> X$, where each morphism $(x,y,g) : x -> y \in \Iso_X(\@M)$ acts as $g : M_x -> M_y$; this action is easily seen to be jointly continuous.
Similarly, $\Iso_X(\@M)$ acts continuously on the fiber powers $M^n_X$.

We will often abuse notation and refer to $(x,y,g) \in \Iso_X(\@M)$ simply as $g : \@M_x \cong \@M_y$.
\end{definition}

\begin{lemma}
\label{thm:isogpd-qpol}
If $\@L$ is a countable language and $p : \@M -> X$ is a second-countable étale $\@L$-structure over a quasi-Polish space $X$, then $\Iso_X(\@M)$ is quasi-Polish.
\end{lemma}
\begin{proof}
First, consider the case where $\@L = \emptyset$, so that $\@M$ is just a second-countable étale space and $\Iso_X(\@M) = \{\text{all bijections between fibers}\}$.
The definition of the topology on $\Iso_X(\@M)$ amounts to embedding it in the fiberwise lower powerspace (\cref{def:lowpow}) of $p \times p : M \times M -> X \times X$ via
\begin{align*}
\Iso_X(\@M) &--> \@F_{X \times X}(M \times M) \\
(x,y,g) &|--> (x,y,\graph(g)).
\end{align*}
So we need only check that the image of this embedding is $\*\Pi^0_2$.
Fix a countable basis $\@S$ of open sections in $M$.
For $(x,y,F) \in \@F_{X \times X}(M \times M)$, $F \subseteq M_x \times M_y$ is the graph of a function iff
\begin{gather*}
\forall S \in \@S\, \paren[\big]{x \in p(S) \implies \exists T \in \@S\, (F \cap (S \times T) \ne \emptyset)}, \\
\forall S, T_1, T_2 \in \@S\, \paren[\big]{F \cap (S \times T_1), F \cap (S \times T_2) \ne \emptyset \implies \exists \@S \ni T_3 \subseteq T_1 \cap T_2\, (F \cap (S \times T_3) \ne \emptyset)}
\end{gather*}
(which say that for each $a \in M_x$, there is at least, resp., at most, one $b \in M_y$ such that $(a,b) \in F$); these conditions are $\*\Pi^0_2$ in the topology on $\@F_{X \times X}(M \times M)$.
Similarly, to say that $F^{-1}$ is the graph of a function is a $\*\Pi^0_2$ condition.

Now if $R \in \@L$ is an $n$-ary relation symbol, then to say that $g : M_x \cong M_y$ preserves $R$ means
\begin{gather*}
\forall \vec{S}, \vec{T} \in \@S^n\, \paren*{
    \begin{aligned}
    &x \in p((S_0 \times_X \dotsb \times_X S_{n-1}) \cap R^\@M)
    \AND g \in \bigcap_i \brbr{S_i |-> T_i} \\
    &\implies y \in p((T_0 \times_X \dotsb \times_X T_{n-1}) \cap R^\@M)
    \end{aligned}
}.
\end{gather*}
Similarly for $g^{-1}$ to preserve $R$, and for preservation of functions.
\end{proof}

\begin{remark}
\label{rmk:isogpd-zpol}
If $X$ and $M$ are both zero-dimensional Polish, then so is $\Iso_X(\@M)$.
Indeed, we may find countable bases for each $M^n_X$ consisting of clopen sections $S \subseteq M^n_X$ such that $p(S) \subseteq X$ is also clopen.
If $S, T \subseteq M^n_X$ are two such sections, then the basic open $\brbr{S |-> T} \subseteq \Iso_X(\@M)$ has open complement $(\dom^{-1}(p(S)) \cap \cod^{-1}(p(T)) => \brbr{S |-> \neg T}) \subseteq \Iso_X(\@M)$.
\end{remark}

\begin{remark}
\label{rmk:isogpd2}
We have the following evident generalization of $\Iso_X(\@M)$:
if $p : \@M -> X$ and $q : \@N -> Y$ are two étale structures over two base spaces, then we have a space $\Iso_{X,Y}(\@M,\@N)$ of isomorphisms between a fiber of $\@M$ and a fiber of $\@N$.
If $\@M, \@N$ are both second-countable over quasi-Polish $X, Y$, then $\Iso_{X,Y}(\@M,\@N)$ is quasi-Polish (zero-dimensional if $X, Y, M, N$ are).
\end{remark}

\section{Étale structures with $\Sigma_1$ saturations}

We now come to the crucial condition on an étale $\@L$-structure, that implies that it is a ``nice'' parametrization of countable $\@L$-structures obeying the usual model-theoretic results.

\begin{definition}
\label{def:estr-osat}
Let $\@M$ be a second-countable étale $\@L$-structure over a space $X$.
We say that $\@M$ \defn{has $\Sigma_1$-definable saturations of open sets}, or more succinctly \defn{has $\Sigma_1$ saturations}, if for every $n \in \#N$ and (basic) open set $U \subseteq M^n_X$, the saturation $\Iso_X(\@M) \cdot U$ under the natural action of $\Iso_X(\@M)$ (recall \cref{def:isogpd}) is equal to $\phi^\@M$ for some $\Sigma_1$ formula $\phi(x_0,\dotsc,x_{n-1})$.
\end{definition}

\begin{remark}
\label{rmk:estr-osat-cts}
From the perspective of \cref{rmk:estr-cts}, this means that
\begin{equation*}
\@M : X --> \{\text{all $\@L$-structures}\}
\end{equation*}
is a ``continuous open map onto its image''.
(This is literally true for a propositional language $\@L$ via \cref{ex:estr-prop}, and is made precise for general first-order $\@L$ by the notion of an \emph{open geometric surjection} between toposes; see \cite[C3.1]{Jeleph}, \cite[IX.6]{MMtopos}.)
\end{remark}

\begin{remark}
\label{rmk:estr-osat-pt}
\Cref{def:estr-osat} may also be regarded as saying that $\@M$ is ``uniformly $\Sigma_1$-atomic'': when $X = 1$, it means that every $\@L_{\omega_1\omega}$-type realized in $\@M$ is axiomatized by a $\Sigma_1$ formula.
\end{remark}

\begin{theorem}
\label{thm:estr-osat-pi02}
If $\@L$ is a countable language and $\@M$ is a second-countable étale $\@L$-structure with $\Sigma_1$ saturations over a quasi-Polish space $X$, then the class of countable $\@L$-structures isomorphic to a fiber of $\@M$ is axiomatizable by a countable $\Pi_2$ theory.
\end{theorem}

This result can be regarded as a common generalization of the well-known special cases alluded to by the two preceding remarks: that a continuous map between (quasi-)Polish spaces which is open onto its image has $\*\Pi^0_2$ image (the combination of \cref{it:qpol-openquot}, \cref{it:qpol-pi02}), and that an atomic model has a $\Pi_2$ Scott sentence (relative to a fragment; see \cite[11.5.7]{Gidst}).
The following proof can be seen as combining these two special cases, namely the proofs of \cref{it:qpol-openquot}, \cref{it:qpol-pi02} in \cite[10.1, 4.1]{Cqpol} and the usual back-and-forth construction of Scott sentences in e.g., \cite[\S3.3]{Mscott}.

\begin{proof}
For each $n \in \#N$, let $\@U_n$ be a countable open basis for the fiber power $M^n_X$.
We use the fact that every quasi-Polish space has a countable posite representation; see \cite[\S8]{Cqpol} for details.
Briefly, each $\vec{a} \in M^n_X$ is determined by its neighborhood filter basis $\@N_{\vec{a}} \subseteq \@U_n$; and this map $\vec{a} |-> \@N_{\vec{a}}$ yields an embedding $M^n_X `-> \#S^{\@U_n}$, which thus has $\*\Pi^0_2$ image \cref{it:qpol-pi02}.
Using that $\@U_n$ is a basis, one can manipulate a $\*\Pi^0_2$ definition of this image into the following form: for a countable binary relation $|>_n$ consisting of certain pairs $(\@V, U)$ of open covers $\@V \subseteq \@U_n$ and $U \in \@U_n$ with $\bigcup \@V = U$, such that
\begin{equation*}
\@V |>_n U \supseteq U' \in \@U_n \implies \exists \@V' |>_N U'\, \forall V' \in \@V'\, \exists V \in \@V\, (V \supseteq V')
\end{equation*}
(``every cover of $U$ refines to one of $U' \subseteq U$''),
we have that $\@A \subseteq \@U_n$ is a neighborhood filter of some $\vec{a} \in M^n_X$ iff it is a $|>_n$-prime filter, meaning a filter satisfying the further $\*\Pi^0_2$ condition
\begin{equation}
\label{eq:estr-osat-pi02:coidl}
\forall \@V |>_n U\, \paren[\big]{U \in \@A \implies \exists V \in \@V\, (V \in \@A)}.
\end{equation}
It then follows that the $\@C \subseteq \@U_n$ which are merely upward-closed and satisfy this condition correspond via \cref{eq:lowpow} to arbitrary closed sets $F \in \@F(M^n_X)$; see \cite[proof of 9.2]{Cqpol}.

For each $U \in \@U_n$, let $\phi_U(x_0,\dotsc,x_{n-1})$ be a $\Sigma_1$ formula defining $\Iso_X(\@M) \cdot U \subseteq M^n_X$.
We claim that the following $\Pi_2$ sentences axiomatize the countable $\@L$-structures isomorphic to a fiber of $\@M$.
For each $n$, let $\pi_n : M^{n+1}_X -> M^n_X$ denote the projection onto the first $n$ coordinates.
\begin{align}
\label{eq:estr-osat-pi02:ax-mono}
\forall \vec{x}\, \paren[\big]{\phi_U(\vec{x}) &-> \phi_V(\vec{x})}
&&\text{for each $U, V \in \@U_n$, $U \subseteq V$}, \\
\label{eq:estr-osat-pi02:ax-coidl}
\forall \vec{x}\, \paren[\big]{\phi_U(\vec{x}) &-> \bigvee_{V \in \@V} \phi_V(\vec{x})}
&&\text{for $\@V |>_n U$}, \\
\label{eq:estr-osat-pi02:ax-hom}
\forall \vec{x}\, \paren[\big]{\phi_U(\vec{x}) &-> \psi(\vec{x})}
&&\text{for basic $n$-ary $\psi$ and $\@U_n \ni U \subseteq \psi^\@M$}, \\
\label{eq:estr-osat-pi02:ax-cohom}
\forall \vec{x}\, \paren[\big]{\phi_U(\vec{x}) \wedge \psi(\vec{x}) &-> \bigvee_{\@U_n \ni V \subseteq U \cap \psi^\@M} \phi_V(\vec{x})}
&&\text{for basic $n$-ary $\psi$}, \\
\label{eq:estr-osat-pi02:ax-ext1}
\forall \vec{x}\, \paren[\big]{\exists y\, \phi_U(\vec{x},y) &<-> \bigvee_{\@U_n \ni V \subseteq \pi_n(U)} \phi_V(\vec{x})}
&&\text{for $U \in \@U_{n+1}$}, \\
\label{eq:estr-osat-pi02:ax-ext2}
\forall \vec{x}, y\, \paren[\big]{\phi_U(\vec{x}) &-> \bigvee_{\@U_{n+1} \ni V \subseteq \pi_n^{-1}(U)} \phi_V(\vec{x},y)}
&&\text{for $U \in \@U_n$}, \\
\label{eq:estr-osat-pi02:ax-ex}
\top &-> \bigvee_{U \in \@U_0} \phi_U.
\end{align}
It is straightforward to check that each $\@M_x$ indeed satisfies these axioms.

Now let $\@N$ be another countable $\@L$-structure satisfying these axioms.
For each $\vec{b} \in N^n$,
\begin{equation*}
\@C_{\vec{b}} := \set{U \in \@U_n}{\phi_U^\@N(\vec{b})}
\end{equation*}
is upward-closed and satisfies \cref{eq:estr-osat-pi02:coidl} by
\cref{eq:estr-osat-pi02:ax-mono},
\cref{eq:estr-osat-pi02:ax-coidl},
thus corresponds to a closed $F_{\vec{b}} \subseteq M^n_X$ such that
\begin{equation*}
F_{\vec{b}} \cap U \ne \emptyset  \iff  U \in \@C_{\vec{b}}  \iff  \phi_U^\@N(\vec{b})
\qquad \forall U \in \@U_n.
\end{equation*}
Using this equivalence, for $\@N$ to satisfy the other 5 axioms above means the following:
\begin{itemize}

\item
\cref{eq:estr-osat-pi02:ax-hom} means
$F_{\vec{b}} \cap \psi^\@M \ne \emptyset \implies \psi^\@N(\vec{b})$
for each basic formula $\psi$, which means that each $\vec{a} \in F_{\vec{b}}$ yields a well-defined partial homomorphism $\@M_x \rightharpoonup \@N$ given by $a_i |-> b_i$.

\item
\cref{eq:estr-osat-pi02:ax-cohom} means
$F_{\vec{b}} \cap U \ne \emptyset \AND \psi^\@N(\vec{b})  \implies  F_{\vec{b}} \cap U \cap \psi^\@M \ne \emptyset$,
i.e., the set of $\vec{a} \in F_{\vec{b}}$ for which the above partial homomorphism $a_i |-> b_i$ is a partial isomorphism is dense $G_\delta$.

\item
\cref{eq:estr-osat-pi02:ax-ext1} means
$(\bigcup_{d \in N} F_{\vec{b},d}) \cap U \ne \emptyset  \iff  F_{\vec{b}} \cap \pi_n(U) \ne \emptyset$,
i.e.,
$\-{\bigcup_{d \in N} F_{\vec{b},d}} = \pi_n^{-1}(F_{\vec{b}}).$

\item
In particular, $\pi_n : M^{n+1}_X -> M^n_X$ restricts to a map $F_{\vec{b},d} -> F_{\vec{b}}$;
\cref{eq:estr-osat-pi02:ax-ext2}
says it has dense image.

\item
\cref{eq:estr-osat-pi02:ax-ex} means $F_\emptyset \ne \emptyset$.

\end{itemize}

Consider now the forest of all $\vec{a} \in F_{\vec{b}}$, each of which determines a partial homomorphism as above, as well as all infinite branches through this forest, which we arrange into a space as follows:
\begin{align*}
Y_n &:= \bigsqcup_{\vec{b} \in N^n} F_{\vec{b}} = \set[\big]{(\vec{b},\vec{a}) \in N^n \times M^n_X}{\vec{a} \in F_{\vec{b}}} \subseteq N^n \times M^n_X, \\
Y_\omega &:= \projlim_{n \in \#N} Y_n = \set[\big]{(\vec{b},\vec{a}) \in N^\#N \times M^\#N_X}{\forall n \in \#N\, ((\vec{b}|n,\vec{a}|n) \in Y_n)}, \\
Y &:= \bigsqcup_{n \le \omega} Y_n.
\end{align*}
For $m \le n \le \omega$, let $\pi_{nm} : Y_n -> Y_m$ denote the projection.
Put a topology on $Y$ with an open basis consisting of, for each $m \in \#N$ and (basic) open $U \subseteq Y_m$ (where $Y_m$ has the disjoint union topology),
\begin{align*}
\Up U := \bigsqcup_{\omega \ge n \ge m} \pi_{nm}^{-1}(U).
\end{align*}
This topology is quasi-Polish; see \cite[A.2]{Cmd}, where it was called the \emph{lax colimit} of the levels $Y_n$.
Since each $\pi_n : F_{\vec{b},d} -> F_{\vec{b}}$ has dense image by \cref{eq:estr-osat-pi02:ax-ext2}, the closure in $Y$ of each $F_{\vec{b}} \subseteq Y_n \subseteq Y$ is all $F_{\vec{b}|m}$ for initial segments $\vec{b}|m$ of $\vec{b}$.
Thus each $\Up Y_n \subseteq Y$ is dense, and so $Y_\omega = \bigcap_n \Up Y_n$ is dense $G_\delta$.
And by \cref{eq:estr-osat-pi02:ax-ex}, $Y_0 = F_\emptyset \ne \emptyset$, i.e., the forest has at least one root.

We now verify that the set of all $(\vec{b},\vec{a}) \in Y$ for which the partial homomorphism $a_i |-> b_i$ is an isomorphism $\@M_x \cong \@N$ is comeager, which will complete the proof:
\begin{itemize}

\item
The set of $(\vec{b},\vec{a}) \in Y$ for which the partial homomorphism $a_i |-> b_i$ is a \emph{partial} isomorphism is dense $\*\Pi^0_2$, since it is the intersection, over all $n \in \#N$, of the union of the lower levels $Y_0, \dotsc, Y_{n-1}$ (which is closed) and $\Up$ of the disjoint union over all $\vec{b} \in N^n$ of those $\vec{a} \in \smash{F_{\vec{b}}}$ for which $a_i |-> b_i$ is a partial isomorphism; and this latter $\Up$ is dense $G_\delta$ in $\Up Y_n$ by \cref{eq:estr-osat-pi02:ax-cohom}.

\item
The set of $(\vec{b},\vec{a}) \in Y$ for which $\vec{b}$ enumerates $N$ is clearly $G_\delta$, and is dense because for a nonempty basic open $\Up (F_{\vec{b}} \cap U) \subseteq Y$ where $U \in \@U_n$, also $\smash{F_{\vec{b},d}} \cap \pi_n^{-1}(U) \ne \emptyset$ by \cref{eq:estr-osat-pi02:ax-ext2}.

\item
Finally, the set of $(\vec{b},\vec{a}) \in Y$ for which $\vec{a} \in M^n_X$ enumerates the fiber $M_x$ in which it lies is dense $\*\Pi^0_2$: we will show that for each open section $S \in \@U_1$, the set of $(\vec{b},\vec{a})$ such that if $\vec{a}$ lies in a fiber over $p(S)$, then $\vec{a}$ includes a point in $S$ (which is a $\*\Pi^0_2$ set), is dense, which suffices because each point in $M$ is in some open section.
For a nonempty basic open $\Up (F_{\vec{b}} \cap U) \subseteq Y$ where $U \in \@U_n$, pick any $\vec{a} \in F_{\vec{b}} \cap U$, say lying in the fiber $M_x$.
If $x \in p(S)$, then $U \times_X S \subseteq M^{n+1}_X$ is an open set such that $F_{\vec{b}} \cap \pi_n(U \times_X S) \ne \emptyset$, since it contains $\vec{a}$.
Thus by the $\Longleftarrow$ direction of \cref{eq:estr-osat-pi02:ax-ext1}, for some $d \in N$, we have $F_{\vec{b},d} \cap (U \times_X S) \ne \emptyset$; an element $(\vec{a}',c)$ of this intersection then belongs to $\Up (F_{\vec{b}} \cap U)$ (since $\vec{a}' \in U$) and includes $c \in S$.
\qedhere

\end{itemize}
\end{proof}

\begin{definition}
If $\@M -> X$ is a second-countable étale structure, and a theory $\@T$ axiomatizes the countable structures isomorphic to some $\@M_x$, we say that \defn{$X$ parametrizes models of $\@T$ via $\@M$}.
\end{definition}

Thus, \cref{thm:estr-osat-pi02} says that if $X$ is quasi-Polish and $\@M$ has $\Sigma_1$ saturations, then $X$ parametrizes models of a $\Pi_2$ theory.
Conversely, every $\Pi_2$ theory has a quasi-Polish parametrization with $\Sigma_1$ saturations; see \cref{ex:estr-per}.

\begin{lemma}
\label{thm:estr-osat-pb}
If $p : \@M -> X$ is a second-countable étale structure with $\Sigma_1$ saturations, and $f : Z -> X$ is a continuous open map, then the pullback structure $f^*(\@M) -> Z$ (\cref{rmk:estr-pb}) also has $\Sigma_1$ saturations.
\end{lemma}

Via \cref{rmk:estr-osat-cts}, this says that ``the composite of open maps $\@M \circ f$ is still open''.

\begin{proof}
Let $\pi_2 : f^*(M) -> M$ be the pullback projection (see the diagram in \cref{def:pb}), which is open because it is a pullback of $f$.
It is easily seen that for $U \subseteq f^*(M)^n_Z$, we have
\begin{equation*}
\Iso_Z(f^*(\@M)) \cdot U = \pi_2^{-1}(\Iso_X(\@M) \cdot \pi_2(U)).
\end{equation*}
Thus if $\phi$ defines $\Iso_X(\@M) \cdot \pi_2(U) \subseteq M^n_X$, then $\phi$ also defines $\Iso_Z(f^*(\@M)) \cdot U$.
\end{proof}

\begin{remark}
\label{thm:estr-osat-res}
If $p : \@M -> X$ is a second-countable étale structure with $\Sigma_1$ saturations, and $Z = \{x \in X \mid \@M_x |= \@T\}$ for a theory $\@T$, then the restriction $\@M|Z$, i.e., pullback of $\@M$ along the inclusion $Z `-> X$, clearly also has $\Sigma_1$ saturations, as witnessed by the same $\Sigma_1$ formulas.

(This is not an instance of the preceding lemma, since $Z \subseteq X$ need not be open.)
\end{remark}

Results such as \ref{thm:estr-osat-pb}, \ref{thm:estr-osat-res} clearly work just as well if ``$\Sigma_1$ saturations'' is replaced with ``$\Sigma_\alpha$ saturations'' throughout, for some $\alpha < \omega_1$.
The following calculation provides a rigorous way of systematically making such generalizations.

\begin{lemma}
\label{thm:estr-osat-morley}
Let $p : \@M -> X$ be a second-countable étale structure, and $\@F$ be a countable fragment of $\@L_{\omega_1\omega}$.
Suppose that saturations of open sets in $\@M$ are definable by formulas which are $\Sigma_1$ over $\@F$.
Then the Morleyized étale structure $p' : \@M' -> X'$ (\cref{def:estr-morley}) has $\Sigma_1$ saturations.
\end{lemma}

\begin{proof}
From \cref{def:estr-morley}, $M'^n_{X'}$ has an open basis of sets of the form
\begin{align*}
U \cap p^{-1}(p(V \cap \phi^\@M))
&= \pi_1(U \times_X (V \cap \phi^\@M)) \\
&= \pi_1((U \times_X V) \cap \phi(x_n, \dotsc, x_{n+m-1})^\@M)
\end{align*}
where $U \subseteq M^n$ is open, $V \subseteq M^m$ is open for some other $m \in \#N$, $\phi$ is a finite conjunction of $m$-ary formulas in $\@F$, $\pi_1 : M^n \times_X M^m -> M^n$ is the first coordinate projection, and $\phi(x_n, \dotsc, x_{n+m-1})^\@M = M^n_X \times_X \phi^\@M$;
since the last is isomorphism-invariant, we get
\begin{align*}
\Iso_{X'}(\@M') \cdot (U \cap p^{-1}(p(V \cap \phi^\@M)))
&= \pi_1\paren[\big]{(\Iso_X(\@M) \cdot (U \times_X V)) \cap \phi(x_n, \dotsc, x_{n+m-1})^\@M} \\
&= \paren[\big]{\exists x_n, \dotsc, x_{n+m-1}\, \psi(x_0, \dotsc, x_{n+m-1}) \wedge \phi(x_n, \dotsc, x_{n+m-1})}^\@M
\end{align*}
where $\psi$ is a formula defining $\Iso_X(\@M) \cdot (U \times_X V) \subseteq M^{n+m}_X$ which is $\Sigma_1$ over $\@F$, hence may be replaced in $\@M'$ by an equivalent $\Sigma_1$ formula in the Morleyized language $\@L'$.
\end{proof}

\begin{corollary}
\label{thm:estr-osat-pi0a}
If $p : \@M -> X$ is a second-countable étale structure over quasi-Polish $X$ with $\Sigma_\alpha$ saturations, then the class of structures isomorphic to one of its fibers is $\Pi_{\alpha+1}$-axiomatizable.
\end{corollary}
This is the $\Sigma_\alpha$ version of \cref{thm:estr-osat-pi02}, and generalizes (U1)$\implies$(U2) of Montalbán's characterization of Scott rank \cite{Mscott} (see \cref{thm:scott} below for other parts of the characterization).
\begin{proof}
Let $\@F$ be the countable fragment generated by $\Sigma_\alpha$ formulas defining the saturations of basic open sets in $\@M$.
By \cref{thm:estr-osat-morley}, the Morleyized $p' : \@M' -> X'$ has $\Sigma_1$ saturations.
Thus by \cref{thm:estr-osat-pi02}, its fibers are $\Pi_2$-axiomatizable in the Morleyized language $\@L'$, hence $\Pi_{\alpha+1}$-axiomatizable in the original language $\@L$.
\end{proof}

\begin{lemma}
\label{thm:estr-osat-isogpd}
If $p : \@M -> X$ is a second-countable étale structure with $\Sigma_1$ saturations, then $\Iso_X(\@M)$ is an \defn{open topological groupoid}, i.e., $\dom, \cod : \Iso_X(\@M) \rightrightarrows X$ are open.
\end{lemma}
\begin{proof}
For open $U, V \subseteq M^n_X$, we have
$\cod(\brbr{U |-> V}) = p((\Iso_X(\@M) \cdot U) \cap V)$;
similarly for $\dom$.
\end{proof}

\begin{remark}
\label{rmk:estr-osat-isogpd2}
More generally, if $p : \@M -> X$ and $q : \@M -> Y$ are two étale structures, and $\@M$ has $\Sigma_1$ saturations, then $\cod : \Iso_{X,Y}(\@M,\@N) -> Y$ (recall \cref{rmk:isogpd2}) is open.
\end{remark}

\section{Examples of parametrizations}
\label{sec:param}

We begin by recasting the standard zero-dimensional Polish parametrization of countably infinite models in terms of an étale structure; see \cite[16.5]{Kcdst}, \cite[\S3.6]{Gidst}.

\begin{example}[space of structures on $\#N$]
\label{ex:estr-infdec}
Let $\@L$ be a countable language,
\begin{align*}
X := \prod_{\text{$n$-ary fn } f \in \@L} \#N^{\#N^n} \times \prod_{\text{$n$-ary rel } R \in \@L} 2^{\#N^n}
\end{align*}
be the zero-dimensional Polish space of $\@L$-structures on $\#N$, with the product topology, and
\begin{align*}
M := X \times \#N
\end{align*}
equipped with the first projection $p : M -> X$; this is clearly étale.
On each fiber $M_x = \{x\} \times \#N \cong \#N$ of $M$, we may put the structure $\@M_x := (M_x,x)$; that is,
\begin{align*}
P^\@M(x, a_0, \dotsc, a_{n-1}) &:= x(P)(a_0, \dotsc, a_{n-1})
\end{align*}
for each symbol (function or relation) $P \in \@L$, using here the ``based'' representation \cref{eq:pb-based} of $M^n_X$.
This defines an étale $\@L$-structure $\@M$ over $X$.
The isomorphism groupoid is
\begin{align*}
\Iso_X(\@M) = \set{(x,y,g) \in X \times X \times S_\infty}{g \cdot x = y}
\end{align*}
where $S_\infty$ is the infinite symmetric group of all bijections $\#N \cong \#N$, acting on $X$ via pushforward of structure (the \defn{logic action}).
Up to topological groupoid isomorphism, we may forget about the second coordinate $y$ above, yielding $\Iso_X(\@M) \cong X \times S_\infty$ (the action groupoid of the logic action).

This $\@M$ does not have $\Sigma_1$ saturations, however, but only $\Sigma_2$ saturations.
The complement of the diagonal in $M^2_X$ is open and $\Iso_X(\@M)$-invariant, but not $\Sigma_1$-definable, since $\ne$ is not $\Sigma_1$; similarly for negations of any relation symbols in $\@L$.
In fact, this is the only issue: by definition of the product topology on $X$, a basic open set in $M^n_X \cong X \times \#N^n$ is of the form
\begin{align*}
\set{(x, a_0, \dotsc, a_{n-1})}{\phi^x(a_n, \dotsc, a_{n+m-1})}
\end{align*}
for some fixed $\vec{a} \in \#N^{n+m}$ and finite conjunction $\phi(y_n, \dotsc, y_{n+m-1})$ of atomic $\@L$-formulas and their negations; the $\Iso_X(\@M)$-saturation of this set is easily seen to be defined by the $\Sigma_2$ formula
\begin{equation*}
\begin{aligned}
\psi(y_0,\dotsc,y_{n-1}) := \exists y_n, \dotsc, y_{n+m-1}\, \paren[\big]{\phi(y_n, \dotsc, y_{n+m-1}) \wedge \bigwedge_{a_i = a_j} (y_i = y_j) \wedge \bigwedge_{a_i \ne a_j} (y_i \ne y_j)}.
\end{aligned}
\end{equation*}
(This is the base case of the proof of the classical Lopez-Escobar theorem; see \cite[16.9]{Kcdst}.)

We may thus Morleyize all negated atomic formulas, as in \cref{ex:morley-neg}, to obtain an expanded language $\@L' \supseteq \@L$ and a $\Pi_2$ $\@L'$-theory $\@T'$.
Let $\@M'$ be the unique expansion of $\@M$ to an étale model of $\@T'$, by interpreting the newly added relation symbols as the complements of the atomic formulas interpreted in $\@M$ (which are clopen).
Then $\@M'$ has $\Sigma_1$ saturations, namely given by the Morleyized formulas $\psi'$ corresponding to the above $\psi$ (as in \cref{def:morley}).
A $\Pi_2$ axiomatization of the fibers of $\@M'$ is given by $\@T'$ together with the ``infinite models'' axioms for each $n \in \#N$
\begin{align}
\label{eq:estr-infdec:ax-inf}
\exists x_0, \dotsc, x_{n-1} \bigwedge_{i \ne j} (x_i \ne x_j).
\end{align}
If we started with a $\Pi_2$ $\@L$-theory $\@T$ which already has only infinite models, then we may restrict this $\@M'$ to the models of $\@T$ (as in \cref{thm:estr-osat-res}) to obtain a parametrization of the models of the Morleyized $\@T$; if furthermore $\@T$ already implied that every negated atomic formula is equivalent to a $\Sigma_1$ formula, then there is no need to Morleyize.
\end{example}

The ``infinite models'' restriction in the above example may be lifted, by enlarging the parametrizing space $X$ to include finite models.
There are several ways to do so, however, that affect the saturation complexity of the resulting $\@M$, i.e., the amount of Morleyization needed to obtain $\Sigma_1$ saturations.
The simplest approach is the following:

\begin{example}[space of structures on $\le \#N$ with $\Delta_1$ sizes]
\label{ex:estr-findec1}
For each $N \le \#N$ (i.e., $N \in \#N \cup \{\#N\}$), define an étale structure $\@M_N -> X_N$ over the space $X_N$ of $\@L$-structures on $N$, exactly as above.
We may then simply take the disjoint unions $X := \bigsqcup_{N \le \#N} X_N$ and $\@M := \bigsqcup_{N \le \#N} \@M_N$.
The groupoid $\Iso_X(\@M)$ consists of the disjoint union of the logic actions of each symmetric group $S_N$ on $X_N$.

However, each $X_N \subseteq X = M^0_X$ is clopen invariant, while ``the model has size $N$'' is $\Sigma_3$-definable for finite $N$ (by a conjunction of a sentence \cref{eq:estr-infdec:ax-inf} and a negation of such a sentence) and $\Pi_3$-definable for infinite $N$ (by the infinite conjunction of all \cref{eq:estr-infdec:ax-inf}); thus this $\@M$ only has $\Sigma_4$ saturations.
Even if we first Morleyize $\ne$, i.e., use the traditional $\Sigma_1$ that includes $\ne$, we still only get $\Sigma_3$ saturations.
(It is easy to see that Morleyizing the above sentences \cref{eq:estr-infdec:ax-inf}, their subformulas, their negations, and their conjunction, thereby rendering each $X_N \subseteq X$ $\Sigma_1$-definable, is sufficient to yield $\Sigma_1$ saturations, since each $\@M_N$ individually has $\Sigma_1$ saturations by the same argument as in the preceding example.)
\end{example}

The worst feature of this approach is that infinite models are clopen in $X$ but only $\Pi_3$-definable.
In order to reduce this gap, we need a coarser topology on the disjoint union $X = \bigsqcup_N X_N$ that allows finite models to converge to infinite ones.
In the next few examples, we assume for simplicity that $\@L$ is relational; there is little loss of generality in doing so, since a $\Pi_2$ theory may be used to impose that a relation is the graph of a function.

\begin{example}[space of structures on $\le \#N$ with $\Delta_1$ finite sizes]
\label{ex:estr-findec2}
Let $\-{\#N} := \{0, 1, 2, \dotsc, \#N\}$, with the topology of the one-point compactification of $\#N$, i.e., $0, 1, 2, \dotsc$ are isolated points converging to $\#N$.
We may realize the disjoint union $X$ in the preceding example as
\begin{align*}
X :={}& \{(N, x) \mid N \in \-{\#N} \AND x \in X_N\} \quad \text{where $X_N$ is as above} \\
={}& \set[\big]{(N, x(R))_{R \in \@L} \in \-{\#N} \times \prod_{R \in \@L} 2^{\#N^n}}{\forall \text{ $n$-ary } R \in \@L\, (x(R) \subseteq N^n \subseteq \#N^n)},
\end{align*}
regarded now as a closed subspace of the compact zero-dimensional space $\-{\#N} \times \prod_{R \in \@L} 2^{\#N^n}$, and
\begin{gather*}
M := \{(N, x, a) \in X \times \#N \mid a \in N\} \subseteq \-{\#N} \times \prod_{R \in \@L} 2^{\#N^n} \times \#N, \\
R^\@M(N, x, a_0, \dotsc, a_{n-1}) := x(R)(a_0, \dotsc, a_{n-1}) \quad \text{for $R \in \@L$}.
\end{gather*}
An open set in $\-{\#N}$ is some finite Boolean combination of the sets $\{n, n+1, \dotsc, \#N\}$ which are axiomatized by the sentences \cref{eq:estr-infdec:ax-inf}.
Thus, a basic open set in $M^n_X$ is
\begin{align}
\label{eq:estr-findec2:basic}
\set[\big]{(N, x, a_0, \dotsc, a_{n-1})}{(N,x) \in X \AND a_0, \dotsc, a_{n-1} \in N \AND \phi^{(N,x)}(a_n, \dotsc, a_{n+m-1}) \wedge \psi^N}
\end{align}
for some $\vec{a} \in \#N^{n+m}$, finite conjunction $\phi(y_n, \dotsc, y_{n+m-1})$ of atomic $\@L$-formulas and their negations, and finite conjunction $\psi$ of the sentences \cref{eq:estr-infdec:ax-inf} and their negations.
Note that in addition to the condition ``$a_0, \dotsc, a_{n-1} \in N$'', for each non-negated atomic formula in $\phi$, each of its arguments among $a_n, \dotsc, a_{n+m-1}$ is also implicitly required to be in $N$.
(Once we remember these implicit conditions, we may assume $\phi$ contains no equalities, which have constant truth value given $\vec{a}$.)
All of these conditions of the form ``$a_i \in N$'' together mean $N \ge \max_i (a_i+1)$, which can itself be expressed by a single sentence $\theta$ of the form \cref{eq:estr-infdec:ax-inf}.
Thus the saturation of the above set is defined by
\begin{align}
\label{eq:estr-findec2:ax}
\exists y_n, \dotsc, y_{n+m-1}\, \paren[\big]{\phi(y_n, \dotsc, y_{n+m-1}) \wedge \bigwedge_{a_i = a_j} (y_i = y_j) \wedge \bigwedge_{a_i \ne a_j} (y_i \ne y_j)} \wedge \psi \wedge \theta
\end{align}
which is $\Sigma_3$; it suffices to Morleyize negations of atomic formulas and of \cref{eq:estr-infdec:ax-inf}.
As before, if we are only interested in parametrizing models of a $\Pi_2$ theory modulo which negations of atomic formulas and ``the model has size $\le n$'' are already $\Sigma_1$-definable, then there is no need to Morleyize anything.
\end{example}

We may further reduce the need for Morleyization, provided we are willing to coarsen the topology to be quasi-Polish:

\begin{example}[space of structures on $\le \#N$ with $\Sigma_1$ lower bounds on size]
\label{ex:estr-findecs}
Modify the preceding\break example by putting the \emph{Scott} topology on $\-{\#N}$, with open sets $\{n, n+1, \dotsc, \#N\}$.
Then the $\psi$ in \cref{eq:estr-findec2:ax} above will only be a conjunction of the sentences \cref{eq:estr-infdec:ax-inf}, no longer their negations, and so the final formula defining the saturation will be $\Sigma_2$; it suffices to Morleyize negations of atomic formulas.

If we also replace $2$ with $\#S$ in the definition of $X$, then the $\phi$ in \cref{eq:estr-findec2:ax} will become basic (and equality-free); and so we only need to Morleyize $\ne$.
In other words, we obtain a parametrization of any $\Pi_2$ theory modulo which $\ne$ is $\Sigma_1$-definable.
(This parametrization was used in \cite{Cscc}; its restriction to the countably infinite models was used in \cite{CMRpos}, \cite{BFRSVle}.)
\end{example}

\begin{remark}
\label{rmk:estr-findecs-0d}
In fact, the quasi-Polishness is not essential here: by \cref{it:qpol-openquot}, we may find a continuous open surjection $Z ->> X$ from a zero-dimensional Polish $Z$ (e.g., $\sum : 2^\#N ->> \-{\#N}$), and then pull back $\@M -> X$ from above to an étale structure over $Z$, with underlying étale space $\subseteq Z \times \#N$ which is clearly also zero-dimensional.
Thus \emph{every $\Pi_2$ theory modulo which $\ne$ is $\Sigma_1$-definable has a zero-dimensional Polish parametrization via a zero-dimensional étale structure with $\Sigma_1$ saturations}.
\end{remark}

The final step of removing the need to Morleyize $\ne$ requires a more drastic change.
We need to render the diagonal in $M^2_X$ open but not closed, which clearly cannot be achieved by merely coarsening the topology on the base $X$ while keeping the fibers subsets of some fixed set such as $\#N$.

\begin{example}[space of partially enumerated structures]
\label{ex:estr-per}
Let $\@L$ be a countable relational language, and let $\@L' := \{{\sim}\} \sqcup \@L$ where $\sim$ is a binary relation symbol.
Construct as in \cref{ex:estr-findecs} a parametrization $\@M' -> X'$ of all $\@L'$-structures (which only has $\Sigma_2$ saturations, due to $\ne$).
Let $X \subseteq X'$ be the structures in which $\sim$ is an equivalence relation and all relations in $\@L$ are $\sim$-invariant.
Then $\sim^{\@M'}$, which is an open set in $(M')^2_{X'}$, restricted to the fibers over $X$ is an open fiberwise equivalence relation on the étale space $M'|X -> X$, i.e., only equating pairs of elements in the same fiber.
It follows that the quotient $M := (M'|X)/({\sim^{\@M'}}|X) -> X$ is also an étale space, and that the quotient map $M'|X ->> M$ is open.
Let $\@M$ be the étale structure on $M$ descended from $\@M'|X$.

Explicitly, the space $X$ consists of tuples $(N, {\sim}, x)$, consisting of $N \in \-{\#N}$ (with the Scott topology), an equivalence relation $\sim$ on $N$, and a $\sim$-invariant $\@L$-structure $x$ on $N$ (both with the Sierpinski topology), so that $x$ descends to the quotient $N/{\sim}$; this descended structure is then the fiber $\@M_{(N, {\sim}, x)}$ of $\@M -> X$.
It is clear that $\@M$ parametrizes all countable $\@L$-structures up to isomorphism, since for any countable $\@L$-structure, we may enumerate it with an initial segment $N \le \#N$ and then lift the structure to $N$.

We now check that $\@M$ has $\Sigma_1$ saturations.
A basic open set in $M^n_X$ is the image under the quotient map of one in $(M')^n_{X'}$ as in \cref{eq:estr-findec2:basic} restricted to $X$, where (in the notation of that example) $\vec{a}$ is fixed, $\phi$ is a basic equality-free formula and $\psi$ only asserts that ``the model has size $\ge \dotsb$'' as in \cref{ex:estr-findecs}, since we used the Sierpinski topology for $X'$.
We claim that the saturation of such an image in $M^n_X$ is defined by \cref{eq:estr-findec2:ax} with all occurrences of $\sim$ in $\phi$ replaced with $=$ and all $\ne$ replaced with $\top$.
Suppose $([b_0],\dotsc,[b_{n-1}]) \in (N'/{\sim'})^n = M^n_{(N',{\sim'},x')}$ satisfies this modified version of \cref{eq:estr-findec2:ax}, as witnessed by $([b_n],\dotsc,[b_{n+m-1}]) \in (N'/{\sim'})^m = M^m_{(N',{\sim'},x')}$; we must find $N, {\sim}, x$ such that $(N, {\sim}, x, a_0, \dotsc, a_{n-1})$ is in \cref{eq:estr-findec2:basic} and there is an isomorphism of the quotient structures $(N,x)/{\sim} \cong (N',x')/{\sim'}$ mapping $[a_i] |-> [b_i]$ for $i < n$.
Let $N \subseteq \#N$ be an initial segment such that
\begin{enumerate}[label=(\roman*)]
\item  if $N' = \emptyset$, then $N = \emptyset$;
\item  each $a_0, \dotsc, a_{n+m-1} \in N$ (consistent with (i) since $[b_0], \dotsc, [b_{n+m-1}] \in N'/{\sim'}$);
\item  $N$ is at least as big as required by $\psi$ (consistent with (i) since $N'/{\sim'} |= \psi$);
\item  $N \setminus \{a_0, \dotsc, a_{n+m-1}\}$ is at least as big as $(N'/{\sim'}) \setminus \{[b_0], \dotsc, [b_{n+m-1}]\}$.
\end{enumerate}
By virtue of the $\bigwedge_{a_i = a_j} (y_i = y_j)$ in \cref{eq:estr-findec2:ax}, $a_i |-> [b_i]$ is a well-defined function $\{a_0,\dotsc,a_{n+m-1}\} -> \{[b_0],\dotsc,[b_{n+m-1}]\} \subseteq N'/{\sim'}$; extend it to a surjection $h : N ->> N'/{\sim'}$, which is easily seen to be possible using the above conditions.
Now define $c \sim d  \coloniff  h(c) = h(d)$, and lift the rest of the quotient structure on $(N',x')/{\sim'}$ to the structure $x$ on $N$.
Then $h$ descends to an isomorphism $N/{\sim} \cong N'/{\sim'}$ mapping $[a_i] |-> [b_i]$, whence $(N, {\sim}, x, a_0, \dotsc, a_{n-1})$ is in \cref{eq:estr-findec2:basic}.

As usual, by further restricting this construction to the models of a $\Pi_2$ theory $\@T$, we obtain that \emph{every $\Pi_2$ theory has a quasi-Polish parametrization with $\Sigma_1$ saturations}.
As in \cref{rmk:estr-findecs-0d}, we may further pull back to a zero-dimensional Polish $Z ->> X$; however, unlike there, we can no longer assume that the underlying étale space $M$ is also zero-dimensional.
\end{example}

\begin{remark}
\label{rmk:estr-subn}
Minor variations of \cref{ex:estr-findec1,ex:estr-findec2,ex:estr-findecs,ex:estr-per} are to allow arbitrary subsets $N \subseteq \#N$, not just initial segments.
(Such a variation of \cref{ex:estr-findecs} was essentially used in \cite{AFscc}.)
\end{remark}

\begin{remark}[space of totally enumerated structures]
\label{ex:estr-er}
A further variation of \cref{ex:estr-per} is to again require $N = \#N$, as in \cref{ex:estr-infdec} (while still using the Sierpinski topology for relations including the equivalence relation $\sim$).
In the above verification that $\@M$ has $\Sigma_1$ saturations, the possibility of $N \ne \#N$ only comes into play in (i), and is clearly not needed if also $N' = \#N$.
The minor downside is that now $\@M$ only parametrizes nonempty models.
(This was the original parametrization used by Joyal--Tierney \cite[VII~\S3~Th.~1]{JTloc}.)
\end{remark}

Finally, we consider a parametrization of a somewhat different flavor, which only differs from that in \cref{ex:estr-er} for languages \emph{with} function symbols.
This parametrization or a variant thereof is widely used for studying classification of algebraic structures such as groups; see e.g., \cite{Tqiso}.

\begin{example}[space of marked structures]
\label{ex:estr-marked}
Let $\@L$ be an arbitrary countable language, possibly with function symbols.
Let $T$ be the set of $\@L$-terms over countably many variables $a_0, a_1, \dotsc$.
Let $X'$ be the space of enumerated $\@L$-structures as in \cref{ex:estr-er}, where function symbols in $\@L$ are encoded via their graphs, and where the role of the index set $\#N$ is replaced by $T$; let $\@M' -> X'$ be the corresponding étale structure.
Thus a point in $X'$ consists essentially of an equivalence relation ${\sim}$ on $T$ and an $\@L$-structure on $T/{\sim}$.
Let $X \subseteq X'$ be the $\*\Pi^0_2$ subspace where $\sim$ is a congruence with respect to the usual syntactic action of function symbols on terms, and where the quotient structure on $T/{\sim}$ interprets each function symbol via this syntactic action.
In other words,
\begin{align*}
X
&\cong \set[\big]{({\sim},x)}{x \in \prod_{\text{$n$-ary rel } R \in \@L} \#S^{T^n} \AND {\sim} \in \#S^{T^2} \text{ is a congruence on } (T,x)} \\
&\cong \{({\sim},\~x) \mid {\sim} \text{ is a congruence on } T \AND \~x \text{ interprets relation symbols on } T/{\sim}\} \\
&\cong \{\text{$\@L$-structures generated by } a_0, a_1, \dotsc\}.
\end{align*}
Let $\@M := \@M'|X$, so that $\@M_{(\sim,x)} = (T/{\sim}, \~x)$ is said generated structure.

Note that $X \subseteq X'$ is not $\Iso_{X'}(\@M')$-invariant; thus we cannot immediately conclude that $\@M$ has $\Sigma_1$ saturations by \cref{thm:estr-osat-res}.
Nonetheless, this is the case, as we now check.
Chasing through \cref{ex:estr-er}, \cref{ex:estr-per}, and ultimately \cref{eq:estr-findec2:basic}, a basic open set in $M^n_X$ is
\begin{equation}
\label{eq:estr-marked:basic}
\set[\big]{({\sim}, x, [t_0], \dotsc, [t_{n-1}])}{({\sim},x) \in X \AND \phi^x(t_n,\dotsc,t_{n+m-1})}
\end{equation}
for some fixed terms $t_0, \dotsc, t_{n+m-1} \in T$ and basic $(\@L \cup \{\sim\})$-formula $\phi$ not mentioning $=$; by absorbing function applications into $t_n, \dotsc, t_{n+m-1}$, we may assume $\phi$ only mentions relation symbols in $\@L$ and $\sim$.
Let $k \in \#N$ be sufficiently large so that $t_0, \dotsc, t_{n+m-1}$ only mention the variables $a_0, \dotsc, a_{k-1}$.
Then the $\Iso_X(\@M)$-saturation of \cref{eq:estr-marked:basic} is defined by
\begin{equation}
\label{eq:estr-marked:ax}
\exists a_0, \dotsc, a_{k-1}\, \paren[\big]{(x_0 = t_0) \wedge \dotsb \wedge (x_{n-1} = t_{n-1}) \wedge \phi'(t_n, \dotsc, t_{n+m-1})}
\end{equation}
where $\phi'$ is $\phi$ with all occurrences of $\sim$ replaced with $=$.
Indeed, if $\@N$ is a countable nonempty (which the fibers of $\@M$ all are) $\@L$-structure which satisfies this formula under an assignment $\vec{x} |-> \vec{b} \in N^n$, as witnessed by some $h : \{a_0, \dotsc, a_{k-1}\} -> N$, we may extend $h$ to the remaining variables $a_k, a_{k+1}, \dotsc$ to yield an enumeration of $N$, and then to all terms to obtain a surjective homomorphism $h : T ->> N$ with respect to the function symbols in $\@L$, and finally pull back the relations as well as $=$ in $\@N$ along $h$ to obtain $({\sim},x) \in X$ such that $h : T ->> N$ descends to an isomorphism $\@M_{({\sim},x)} \cong \@N$ and $({\sim},x,[t_0],\dotsc,[t_{n-1}])$ is in \cref{eq:estr-marked:basic} where $t_0, \dotsc, t_{n-1}$ are any $h$-preimages of $\vec{b}$.
\end{example}

\begin{remark}
For finite $N \in \#N$, one could analogously construct a space $X_N$ as above starting from only $N$ variables $a_0, \dotsc, a_{N-1}$, parametrizing $N$-generated structures.
However, saturations would only be $\Sigma_3$, since in the defining formulas \cref{eq:estr-marked:ax} (where without loss we can take $k := N$), we need to add a clause $\forall y \bigvee_{t \in T} (y = t)$ saying that $a_0, \dotsc, a_{N-1}$ are generators.
\end{remark}

\section{Omitting types and Scott rank}
\label{sec:omittype}

The classical Baire category theorem states that in sufficiently nice topological spaces, a countable intersection of dense open (or even $\*\Pi^0_2$) sets is still dense.
The omitting types theorem in model theory, and its infinitary variants, state that a countable conjunction of ``dense $\Sigma_1$ (or even $\Pi_2$) formulas'' is still ``dense''.
It is well-known folklore that these two results are closely related, and in fact omitting types can also be reduced to Baire category via some standard coding tricks; see e.g., \cite{ETomit} and the references therein.
As an application of the ``continuous open map'' perspective of \cref{rmk:estr-osat-cts}, we give here an easy and transparent version of such a reduction.

\begin{theorem}[omitting types]
\label{thm:omittype}
Let $\@T$ be a satisfiable countable $\Pi_2$ theory.
\begin{enumerate}[label=(\alph*)]
\item
Let $\phi_i(\vec{x}_i)$ be countably many $\Pi_2$ formulas of various arities $n_i \in \#N$, such that for each $\Sigma_1$ $\theta(\vec{x}_i)$ satisfiable in some model of $\@T$, $\phi_i \wedge \theta$ is also satisfiable in some model of $\@T$.
Then $\@T \cup \{\forall \vec{x}_i\, \phi_i(\vec{x}_i)\}_i$ is satisfiable.
\item
Let $\psi_i(\vec{x}_i)$ be countably many $\Sigma_2$ formulas of various arities $n_i \in \#N$, such that for each $\Sigma_1$ $\theta(\vec{x}_i)$, if $\@T |= \forall \vec{x}_i\, (\theta(\vec{x}_i) -> \psi_i(\vec{x}_i))$, then $\@T |= \forall \vec{x}_i\, \neg \theta(\vec{x}_i)$.
Then $\@T \cup \{\forall \vec{x}_i\, \neg \psi_i(\vec{x}_i)\}_i$ is satisfiable.
\end{enumerate}
\end{theorem}
\begin{proof}
The statements are dual; we prove (a).
By \cref{ex:estr-per}, there is a quasi-Polish parametrization $p : \@M -> X$ of $\@T$ with $\Sigma_1$ saturations, i.e., a ``continuous open map $\@M : X ->> \{\text{models of $\@T$}\}$''.
Then ``the $\@M$-preimage of each dense $\Pi_2$ $\phi_i$ is dense $\*\Pi^0_2$'', i.e., each $\phi_i^\@M \subseteq M^{n_i}_X$ is dense $\*\Pi^0_2$, since for each open $\emptyset \ne U \subseteq M^{n_i}_X$, letting $\theta$ be $\Sigma_1$ with $\theta^\@M = \Iso_X(\@M) \cdot U \supseteq U$, $\theta$ is satisfiable in a model of $\@T$, whence so is $\phi \wedge \theta$, whence $\emptyset \ne \phi^\@M \wedge \theta^\@M = \Iso_X(\@M) \cdot (\phi^\@M \cap U)$ since $\phi^\@M$ is $\Iso_X(\@M)$-invariant.
Hence, each $\phi_i^\@M \subseteq M^{n_i}_X$ is comeager.
It follows that each $(\forall \vec{x}_i\, \phi_i(\vec{x}_i))^\@M \subseteq X$ is comeager, since dually, the $p$-image of each meager set in $M^{n_i}_X$ is easily seen to be meager in $X$, by considering a countable basis of open sections in $M^{n_i}_X$.
Thus $\bigcap_i (\forall \vec{x}_i\, \phi_i(\vec{x}_i))^\@M$ is also comeager, hence dense by Baire category; picking any $x$ in this set, we have $\@M_x |= \@T \cup \{\forall \vec{x}_i\, \phi_i(\vec{x}_i)\}_i$.
\end{proof}

\begin{example}
Taking $\@T$ to be a propositional theory, as in \cref{ex:prop} (and the arities $n_i$ above to all be $0$), we recover the Baire category theorem for quasi-Polish spaces as a special case.
\end{example}

\begin{remark}
Via Morleyization, \cref{thm:omittype} generalizes to an omitting types theorem for an arbitrary countable fragment $\@F$, which generalizes the classical omitting types for $\@L_{\omega_1\omega}$ (see e.g., \cite[4.9]{Mlo1o}) to our more general fragments (\cref{def:morley}) not necessarily closed under $\neg, \wedge, \vee, \forall, \exists$.
\end{remark}

In particular, \cref{thm:omittype} is sufficiently general to imply Montalbán's omitting types theorem for $\Pi_\alpha$ \cite[3.2]{Mscott}, using which we may generalize most parts of Montalbán's characterization of Scott rank \cite[1.1]{Mscott} to our weaker notions of $\Sigma_\alpha, \Pi_\alpha$:

\begin{theorem}[characterization of Scott rank]
\label{thm:scott}
Let $p : \@M -> X$ be a second-countable étale structure with $\Sigma_1$ saturations over a quasi-Polish space $X$, and let $x \in X$.
The following are equivalent:
\begin{enumerate}[label=(\roman*)]
\item \label{thm:scott:orbit}
Every automorphism orbit of $\@M_x$ is $\Sigma_\alpha$-definable without parameters.
\item \label{thm:scott:scott}
$\@M_x$ has a $\Pi_{\alpha+1}$ Scott sentence $\phi$.
\item \label{thm:scott:iso}
$\Iso_X(\@M) \cdot x \subseteq X$ is $\*\Pi^0_{\alpha+1}$.
\item \label{thm:scott:type}
For every $\Pi_\alpha$ formula $\psi(\vec{z})$ with $\psi^{\@M_x} \ne \emptyset$, there is a $\Sigma_\alpha$ formula $\theta(\vec{z})$ with $\emptyset \ne \theta^{\@M_x} \subseteq \psi^{\@M_x}$.
\end{enumerate}
\end{theorem}
To recover \cite[1.1]{Mscott}, Morleyize negated atomic formulas as in \cref{ex:morley-neg} and take the standard parametrization from \cref{ex:estr-infdec}.
The proof is the same as in \cite{Mscott} but with results used therein replaced by their generalizations from this paper, namely \cref{thm:estr-osat-pi0a} for \cref{thm:scott:orbit}$\implies$\cref{thm:scott:scott}, the Lopez-Escobar \cref{thm:lopez-escobar} for \cref{thm:scott:scott}$\iff$\cref{thm:scott:iso}, and \cref{thm:omittype} for \cref{thm:scott:scott}$\implies$\cref{thm:scott:type}.

\begin{remark}
The omitting types theorem has an analog in topos theory, called the \emph{$\neg\neg$-subtopos}, which is the ``theory of models of $\@T$ omitting all non-isolated $\Pi_1$ types''; see \cite[A4.5.9]{Jeleph}.
\end{remark}

\section{Imaginary sorts and the Joyal--Tierney theorem}
\label{sec:jt}

\begin{definition}
For an étale structure $p : \@M -> X$, an \defn{étale $\Iso_X(\@M)$-space over $X$} is an étale space $q : A -> X$ equipped with a (jointly) continuous action of the topological groupoid $\Iso_X(\@M)$, where each $g : \@M_x \cong \@M_y \in \Iso_X(\@M)$ acts via a bijection $A_x \cong A_y$.
\end{definition}

The prototypical étale $\Iso_X(\@M)$-space is $M$.
We may build other étale $\Iso_X(\@M)$-spaces via the following basic operations:
\begin{itemize}

\item
Each fiber power $M^n_X$ is also an étale $\Iso_X(\@M)$-space.
More generally, a fiber product of finitely many étale $\Iso_X(\@M)$-spaces is again such a space, under the diagonal action.

\item
If $A -> X$ is an étale $\Iso_X(\@M)$-space, and $U \subseteq A$ is an open set invariant under the action of $\Iso_X(\@M)$, then $U$ is itself an étale $\Iso_X(\@M)$-space.
If $A = M^n_X$, and $\@M$ has $\Sigma_1$ saturations, then such $U$ are precisely the $\Sigma_1$-definable $n$-ary relations $\phi^\@M \subseteq M^n_X$.

\item
A (possibly infinite) disjoint union of étale $\Iso_X(\@M)$-spaces is again such.

\item
If $p : A -> X$ is an étale $\Iso_X(\@M)$-space, and ${\sim} \subseteq A \times_X A$ is an $\Iso_X(\@M)$-invariant open (in $A \times_X A$) equivalence relation, then $p$ and the action descend to the quotient space $A/{\sim}$ which again becomes an étale $\Iso_X(\@M)$-space.

\end{itemize}
We now introduce syntactic names for these last two operations, generalizing formulas:

\begin{definition}
\label{def:imag}
For a countable $\@L_{\omega_1\omega}$ theory $\@T$, a \defn{$\Sigma_1$ imaginary sort over $\@T$} is an expression
\begin{align*}
\Phi = (\bigsqcup_i \phi_i)/(\bigsqcup_{i,j} \epsilon_{ij})
\end{align*}
where $i, j$ run over the same countable index set, each $\phi_i$ is a $\Sigma_1$ formula in some number of free variables $n_i \in \#N$, each $\epsilon_{ij}$ is a $\Sigma_1$ formula in $n_i+n_j$ variables, and $\@T$ proves the sentences
\begin{align*}
\forall \vec{x}, \vec{y}\, \paren[\big]{\epsilon_{ij}(\vec{x},\vec{y}) &-> \phi_i(\vec{x}) \wedge \phi_j(\vec{y})}, \\
\forall \vec{x}\, \paren[\big]{\phi_i(\vec{x}) &-> \epsilon_{ii}(\vec{x},\vec{x})}, \\
\forall \vec{x}, \vec{y}\, \paren[\big]{\epsilon_{ij}(\vec{x},\vec{y}) &-> \epsilon_{ji}(\vec{y},\vec{x})}, \\
\forall \vec{x}, \vec{y}, \vec{z}\, \paren[\big]{\epsilon_{ij}(\vec{x},\vec{y}) \wedge \epsilon_{jk}(\vec{y},\vec{z}) &-> \epsilon_{ik}(\vec{x},\vec{z})}.
\end{align*}
Equivalently by the completeness theorem, in each $\@M |= \@T$,
\begin{align*}
\bigsqcup_{i,j} \epsilon_{ij}^\@M \subseteq \bigsqcup_{i,j} M^{n_i+n_j} \cong (\bigsqcup_i M^{n_i})^2
\text{ is an equivalence relation on }
\bigsqcup_i \phi_i^\@M \subseteq \bigsqcup_i M^{n_i}.
\end{align*}
We then define the \defn{interpretation of $\Phi$ in $\@M$} to be the quotient set
\begin{align*}
\Phi^\@M := (\bigsqcup_i \phi_i^\@M) / (\bigsqcup_{i,j} \epsilon_{ij}^\@M).
\end{align*}
Similarly, if $p : \@M -> X$ is an étale model of $\@T$, then we have an étale $\Iso_X(\@M)$-space $\Phi^\@M -> X$ defined in the same way, as the quotient of $\bigsqcup_i \phi_i^\@M -> X$ by the open equivalence relation $\bigsqcup_{i,j} \epsilon_{ij}^\@M$.
\end{definition}

\begin{remark}
\label{rmk:imag-spretop}
Thus, a $\Sigma_1$ imaginary is a syntactic name for a \emph{quotient} of a \emph{countable disjoint union} of \emph{$\Sigma_1$-definable subsets} of \emph{finite products} $M^n$ of copies of the underlying set or étale space $M$ of an (étale) model $\@M$.
We may always distribute these 4 types of operations over each other to put them into this order; thus, imaginaries are themselves closed under these operations:
\begin{itemize}

\item
If $\Phi = (\bigsqcup_i \phi_i)/(\bigsqcup_{i,j} \epsilon_{ij})$ and $\Psi = (\bigsqcup_k \psi_k)/(\bigsqcup_{k,l} \eta_{kl})$ are imaginaries, we may define their \defn{product sort} $\Phi \times \Psi := (\bigsqcup_{i,k} \phi_i \wedge \psi_k)/(\bigsqcup_{i,j,k,l} \epsilon_{ij} \wedge \eta_{kl})$.
There is also the \defn{singleton sort} (nullary product) $\*1 := \top/\top$ where here $\top$ is regarded as having no free variables.

\item
If $\Phi = (\bigsqcup_i \phi_i)/(\bigsqcup_{i,j} \epsilon_{ij})$ is a $\Sigma_1$ imaginary sort over $\@T$, a \defn{$\Sigma_1$-definable subsort} $\Psi \subseteq \Phi$ is given by countably many $\Sigma_1$ formulas $\psi_i$ such that ``$\bigsqcup_i \psi_i \subseteq \bigsqcup_i \phi_i$ is $\@T$-provably $(\bigsqcup_{i,j} \epsilon_{ij})$-invariant'', which can be expressed by certain $\Pi_2$ sentences much as in \cref{def:imag}.
Such $\Psi$ can also be thought of as a $\Sigma_1$ imaginary in its own right, namely $\Psi = (\bigsqcup_i \psi_i)/(\bigsqcup_{i,j} \epsilon_{ij} \wedge \psi_i \wedge \psi_j)$, such that $\Psi^\@M \subseteq \Phi^\@M$ for every $\@M |= \@T$.

\item
A \defn{countable disjoint union} of imaginaries $\bigsqcup_i \Phi_i$ is given by simply merging the formulas.

\item
Finally, if $\Phi = (\bigsqcup_i \phi_i)/(\bigsqcup_{i,j} \epsilon_{ij})$ is an imaginary, and $H \subseteq \Phi \times \Phi$ is a $\Sigma_1$-definable subsort which is $\@T$-provably an equivalence relation, then we may define the \defn{quotient sort} $\Phi/H$ by replacing the $\epsilon_{ij}$'s with the formulas $\eta_{ij}$ defining $H$.
\end{itemize}
\end{remark}

\begin{theorem}[Joyal--Tierney \cite{JTloc}]
\label{thm:jt}
Let $\@T$ be a countable $\@L_{\omega_1\omega}$ theory, $p : \@M -> X$ be a second-countable étale structure with $\Sigma_1$ saturations parametrizing models of $\@T$.
Then every second-countable étale $\Iso_X(\@M)$-space $q : A -> X$ is isomorphic to $\Phi^\@M$ for some $\Sigma_1$ imaginary $\Phi$ over $\@T$.
\end{theorem}
The following proof is a streamlined version of the proofs in \cite[\S1.4]{AFscc} and \cite[8.1]{Cscc} for two specific parametrizations (namely those in \cref{rmk:estr-subn,ex:estr-findecs}).
\begin{proof}
Let $a \in A_x$, and let $a \in S \subseteq A$ be an open section.
Since the identity $1_{\@M_x} \in \Iso_X(\@M)$ fixes $a$, by continuity of the action, there are open neighborhoods $a \in S' \subseteq S$ and $1_{\@M_x} \in \brbr{T |-> T'} \subseteq \Iso_X(\@M)$, where we may assume $T, T' \subseteq M^n_X$ are open sections, such that $\brbr{T |-> T'} \cdot S' \subseteq S$.
Since $1_{\@M_x} \in \brbr{T |-> T'}$, we have $x \in p(T \cap T')$.
By replacing $S$ with $S' \cap q^{-1}(p(T \cap T'))$ and $T, T'$ with $T \cap T' \cap p^{-1}(q(S'))$, we get $p(T) = q(S)$ and $\brbr{T |-> T} \cdot S \subseteq S$.

So we have shown that $A$ has a basis of open sections $S \subseteq A$ for which there exists an open section $T \subseteq M^n_X$ for some $n$ such that $p(T) = q(S)$ and $\brbr{T |-> T} \cdot S \subseteq S$.
For such $S, T$, consider
\begin{align*}
h_{S,T} : \Iso_X(\@M) \cdot T &--> A \\
(g : \@M_y \cong \@M_z) \cdot T_y &|--> g \cdot S_y.
\end{align*}
(Here we are abusing notation by writing $T_y$ for the unique element of $T_y$, etc.)
We claim that this is a well-defined, continuous $\Iso_X(\@M)$-equivariant map, whose image contains $S$.
Well-definedness is because if $T_y \ne \emptyset$ then $S_y \ne \emptyset$ as $p(T) \subseteq q(S)$, and if $g \cdot T_y = g' \cdot T_{y'}$, where $g : \@M_y \cong \@M_z$ and $g' : \@M_{y'} \cong \@M_z$ and both $y, y' \in p(T)$, then $g^{\prime-1} \circ g : \@M_y \cong \@M_{y'} \in \brbr{T |-> T}$ and so $g \cdot S_y = g' \cdot S_{y'}$.
Equivariance is clear; and $S \subseteq \im(h_{S,T})$ by taking $g = 1$, using $p(T) = q(S)$.
For continuity: suppose
\begin{equation*}
h_{S,T}(g \cdot T_y) = g \cdot S_y \in S',
\end{equation*}
where $S' \subseteq A$ is another open section for which there is an open section $T' \subseteq M^{n'}_X$ such that $p(T') = q(S') \ni z$ and $\brbr{T' |-> T'} \cdot S' \subseteq S'$.
Let $g^{-1} \cdot T'_z \in T'' \subseteq M^{n'}_X$ be another open section, and let
\begin{equation*}
U := q(S \cap (\brbr{T' |-> T''} \cdot S')) \subseteq X.
\end{equation*}
Then $S_y = g^{-1} \cdot S'_z \in \brbr{T' |-> T''} \cdot S'$ whence $y \in U$ whence $g \cdot T_y \in \brbr{T'' |-> T'} \cdot (T \cap p^{-1}(U))$; and
\begin{align*}
h_{S,T}(\brbr{T'' |-> T'} \cdot (T \cap p^{-1}(U)))
&= \brbr{T'' |-> T'} \cdot (S \cap q^{-1}(U)) \\
&= \brbr{T'' |-> T'} \cdot (S \cap (\brbr{T' |-> T''} \cdot S')) \quad \text{since $q|S$ is injective} \\
&\subseteq \brbr{T' |-> T'} \cdot S' \subseteq S'.
\end{align*}

Now since $A$ is second-countable, hence Lindelöf, it has a countable cover by open sections $S_i \subseteq A$ for which there exist corresponding $T_i \subseteq M^{n_i}_X$ as above.
For each $i$, let $\phi_i$ be a $\Sigma_1$ formula defining $\Iso_X(\@M) \cdot T_i \subseteq M^{n_i}_X$.
We then have a continuous equivariant surjection
\begin{align*}
h := \bigsqcup_i h_{S_i,T_i} : \bigsqcup_i \phi_i^\@M -->> A.
\end{align*}
Its congruence kernel ${\sim} = \{(\vec{a},\vec{b}) \in (\bigsqcup_i \phi_i^\@M)^2_X \mid h(\vec{a}) = h(\vec{b})\} \subseteq (\bigsqcup_i M^{n_i}_X)^2_X \cong \bigsqcup_{i,j} M^{n_i+n_j}_X$ is an invariant open set, thus there are $\Sigma_1$ formulas $\epsilon_{ij}$ defining ${\sim}$ restricted to each $M^{n_i+n_j}_X$.
Then $\Phi := (\bigsqcup_i \phi_i)/(\bigsqcup_{i,j} \epsilon_{ij})$ is a $\Sigma_1$ imaginary over $\@T$, because the interpretation of $\bigsqcup_{i,j} \epsilon_{ij}$ in each countable model of $\@T$, which is isomorphic to some fiber $\@M_x$, is the equivalence relation $\sim_x$ on $\bigsqcup_i \phi_i^{\@M_x}$.
And $h$ descends to an isomorphism of étale $\Iso_X(\@M)$-spaces
$\Phi^\@M = (\bigsqcup_i \phi_i^\@M)/{\sim} \cong A$.
\end{proof}

\begin{example}
If $X = 1$ and $\@M$ is a countable structure with $\Sigma_1$-definable automorphism orbits, as in \cref{rmk:estr-osat-pt}, then \cref{thm:jt} says that every continuous action of $\Aut(\@M)$ on a countable discrete space $A$ is named by some $\Sigma_1$ imaginary of $\@M$.
To see this directly: for every $a \in A$, the stabilizer $\Aut(\@M)_a \subseteq \Aut(\@M)$ is a clopen subgroup, hence contains $\Aut(\@M,\vec{b})$ for finitely many constants $\vec{b} \in M^n$; then the orbit $\Aut(\@M) \cdot a$ is a definable quotient of the $\Sigma_1$-definable orbit $\Aut(\@M) \cdot \vec{b}$.
The above proof can be seen as the natural generalization of this to étale structures.
\end{example}

\begin{remark}
\label{rmk:jt}
\Cref{thm:jt} says that the map
\begin{align*}
\{\text{$\Sigma_1$ imaginaries over $\@T$}\} &\overset{\sim}{-->} \{\text{second-countable étale $\Iso_X(\@M)$-spaces}\} \\
\Phi &|--> \Phi^\@M
\end{align*}
is surjective up to isomorphism.
The full Joyal--Tierney theorem says that it is in fact an equivalence of categories; \cref{thm:jt} above is merely the key ingredient.
Namely, the category on the left has:
\begin{itemize}
\item  $\Sigma_1$ imaginary sorts $\Phi$ as objects;
\item  subobjects of $\Phi$ are $\Sigma_1$-definable subsorts $\Psi \subseteq \Phi$ (\cref{rmk:imag-spretop}), which (modulo $\@T$-provable equivalence) correspond bijectively to open invariant subsets of $\Phi^\@M$ by $\Sigma_1$ saturations;
\item  morphisms $\Theta : \Phi -> \Psi$ are \defn{$\Sigma_1$-definable functions}, i.e., $\Sigma_1$-definable subsorts $\Theta \subseteq \Phi \times \Psi$ which are $\@T$-provably the graph of a function.
We usually identify definable functions modulo $\@T$-provable equivalence.
\end{itemize}
For a countable $\Pi_2$ theory $\@T$, this category of $\Sigma_1$ imaginaries is called the \defn{syntactic $\sigma$-pretopos} (also known as \defn{classifying $\sigma$-pretopos}) of the theory $\@T$, and can be thought of as a canonical algebraic representation of the syntax of $\@T$; see \cite[D1.4]{Jeleph}, \cite[\S10]{Cscc}.
Now the above equivalence boils down to the standard category-theoretic fact \cite[D3.5.6]{Jeleph} that a functor preserving finite categorical limits is an equivalence iff it is surjective on objects up to isomorphism and bijective on subobjects (in the usual categorical parlance, \emph{essentially surjective}, \emph{conservative}, and \emph{full on subobjects}).
\end{remark}

\section{Interpretations and functors}
\label{sec:interp}

Taking a more global view, the equivalence of categories in \cref{rmk:jt} says that the map
\begin{align*}
\yesnumber
\label{eq:interp0}
\{\text{countable $\Pi_2$ theories}\} &`--> \{\text{(quasi-)Polish groupoids}\} \\
\@T &|--> \Iso(\text{some parametrization of $\@T$})
\end{align*}
is ``injective'': we may recover a theory from the groupoid of its parametrization.
To make this statement precise, we need to discuss the appropriate structure on the class of all theories.

\begin{definition}
\label{def:interp}
Let $\@T_1, \@T_2$ be countable $\Pi_2$ theories in two respective countable languages $\@L_1, \@L_2$.
A \defn{$\Sigma_1$ interpretation} $\@F : \@T_1 -> \@T_2$ is a ``model of $\@T_1$ within the category of $\Sigma_1$ imaginaries over $\@T_2$'':
\begin{itemize}
\item
First, one defines an interpretation $\@F : \@L_1 -> \@T_2$ to be an ``$\@L_1$-structure in $\@T_2$-imaginaries'', consisting of an underlying $\Sigma_1$ imaginary $F$ over $\@T_2$, an interpretation of each $n$-ary $R \in \@L_1$ as a $\Sigma_1$-definable subsort $R^\@F \subseteq F^n$, and an interpretation of each $n$-ary $f \in \@L_1$ as a $\Sigma_1$-definable function $f^\@F : F^n -> F$.
\item
Next, one inductively defines, in the usual way, interpretations of $\@L_1$-terms and formulas as definable functions and subsorts over $\@T_2$.
\item
One then says that $\@F$ satisfies an $\@L_1$-sentence $\phi \in \@T_1$ if its interpretation $\phi^\@F$, which will be a definable subsort of the singleton sort $\*1$ over $\@T_2$ (recall \cref{rmk:imag-spretop}), is the entirety of $\*1$.
\item
Finally, one can extend the interpretation of $\@L_1$-formulas to imaginaries $\Phi$ over $\@T_1$, which become interpreted as imaginaries $\Phi^\@F$ over $\@T_2$.
\end{itemize}
The categorical viewpoint of \cref{rmk:jt} is particularly useful when dealing with interpretations: a $\Sigma_1$ interpretation $\@F : \@T_1 -> \@T_2$ is simply a functor $\Phi |-> \Phi^\@F$ from the syntactic $\sigma$-pretopos of $\@T_1$ to the syntactic $\sigma$-pretopos of $\@T_2$, preserving finite limits and countable colimits.
For instance, this makes it easy to define the \emph{composition} of interpretations $\@T_1 -> \@T_2 -> \@T_3$.
See \cite[\S10]{Cscc}.

A \defn{$\Sigma_1$-definable isomorphism} $h : \@F \cong \@G : \@T_1 -> \@T_2$ between two interpretations is defined in the same way as a usual isomorphism, but replacing the underlying function with a $\Sigma_1$-definable function over $\@T_2$.
Equivalently, it is a natural isomorphism between functors.

An interpretation
\begin{align*}
\@F : \@T_1 &--> \@T_2 \\
\intertext{can be thought of as a syntactic specification for a map}
\yesnumber
\label{eq:interp-mod}
\{\text{models of $\@T_1$}\} &<-- \{\text{models of $\@T_2$}\} : \@F^* \\
\@F^\@M &<--| \@M.
\end{align*}
Namely, given a (countable/étale) $\@M |= \@T_2$, the model $\@F^*(\@M) = \@F^\@M |= \@T_1$ has underlying set/étale space $F^\@M$ and each symbol $P \in \@L_1$ interpreted as $(P^\@F)^\@M$.
Similarly, a $\Sigma_1$-definable isomorphism $h : \@F \cong \@G$ between interpretations yields, for each $\@M |= \@T_2$, an $\@L_1$-isomorphism $h^\@M : \@F^\@M \cong \@G^\@M$.
\end{definition}

For detailed background on interpretations and imaginaries (albeit in the context of finitary first-order logic), see \cite[Ch.~5]{Hmod}.
We first give a familiar finitary example:

\begin{example}
The construction from each integral domain $\@R$ of its field of fractions is specified by an interpretation
$\@F : (\text{theory of fields}) -> (\text{theory of integral domains})$:
\begin{itemize}
\item  The underlying imaginary is $F = \phi/\epsilon$ where $\phi, \epsilon$ are the formulas in the language of rings:
\begin{align*}
\phi(x,y) &:= (y \ne 0), &
\epsilon(x_1,y_1,x_2,y_2) := (x_1y_2 = x_2y_1).
\end{align*}
Thus given an integral domain $\@R$, the field of fractions $\@F^\@R$ has underlying set
\begin{align*}
F^\@R = \{(a,b) \in R^2 \mid b \ne 0\}/{\sim} \quad \text{where} \quad (a_1,b_1) \sim (a_2,b_2) \coloniff a_1b_2 = a_2b_1.
\end{align*}
\item  The operation $+$ of fields is interpreted as the definable function $+^\@F = \psi/\epsilon^3$ where
\begin{align*}
\psi(x_1,y_1,x_2,y_2,x_3,y_3) &:= \epsilon(x_1y_2+x_2y_1, y_1y_2, x_3, y_3).
\end{align*}
In other words, for an integral domain $\@R$, this defines the graph of $+$ on $\@F^\@R$:
\begin{align*}
[(a_1,b_1)] + [(a_2,b_2)] = [(a_3,b_3)]  \iff  (a_1b_2+a_2b_1, b_1b_2) \sim (a_3,b_3).
\end{align*}
\end{itemize}
Similarly for the other operations.
(This is not a $\Sigma_1$ interpretation, unless we include a unary relation symbol for ``nonzero'' in the language of integral domains.)
\end{example}

For an example that uses the availability of countable disjoint unions:

\begin{example}
We have an interpretation from the theory of groups to the theory of sets, that specifies the construction from a set $X$ of the free group over $X$.
Its underlying imaginary is a countable disjoint union $\bigsqcup_{\vec{s} \in \{\pm1\}^{<\omega}} \phi_s$ where for each $\vec{s} = (s_0,\dotsc,s_{n-1})$,
the formula $\phi_{\vec{s}}(x_0,\dotsc,x_{n-1}) := \bigwedge_{s_i \ne s_{i+1}} (x_i \ne x_{i+1})$ says ``$x_0^{s_0} \dotsm x_{n-1}^{s_{n-1}}$ is a reduced word in the free group''.
\end{example}

\begin{example}
The notion of $\Sigma_1$ interpretation subsumes that of étale structure, which is essentially the same thing as an interpretation into a propositional theory.
Indeed, recall from \cref{ex:prop} that a quasi-Polish space $X$ is the space of models of a countable $\Pi_2$ propositional theory $\@T_0$.
A second-countable étale space $p : A -> X$ then corresponds to a $\Sigma_1$ imaginary over $\@T_0$ (this is the trivial case of the Joyal--Tierney \cref{thm:jt}, for the trivial étale structure $X -> X$).
Thus, a $\Sigma_1$ interpretation $\@M$ from another countable $\Pi_2$ theory $\@T$ into $\@T_0$ is a model of $\@T$ in $\@T_0$-imaginaries, i.e., étale spaces over $X$, i.e., an étale model of $\@T$ over $X$.
The induced map $\@M^*$ as in \cref{eq:interp-mod}, from models $x \in X$ of $\@T_0$ to models of $\@T$, is just the ``continuous map $x |-> \@M_x$'' from \cref{rmk:estr-cts}.
\end{example}

\begin{remark}
\label{rmk:2cat}
Thus, a natural structure to put on the class of all countable $\Pi_2$ theories, on the left-hand side of \cref{eq:interp0}, is that of a \defn{2-category}, whose objects are countable $\Pi_2$ theories $\@T$, morphisms are $\Sigma_1$ interpretations $\@F : \@T_1 -> \@T_2$, and 2-cells are $\Sigma_1$-definable isomorphisms%
\footnote{It makes sense to consider more generally non-invertible 2-cells which are definable homomorphisms; however, in order to reflect these on the semantic (right) side of \cref{eq:interp0}, one should then replace $\Iso_X(\@M)$ with the \emph{homomorphism category}, a topological category, which is a more involved notion than a topological groupoid.
See \cite{Mtop2}.}
between interpretations $h : \@F \cong \@G : \@T_1 -> \@T_2$.
We may visualize this 2-category as follows:
\begin{equation*}
\begin{tikzcd}
\@T_1 \rar[bend left,"\@F"] \rar[bend right,"\@G"'] \rar[phantom,"\scriptstyle\Down h"] &
\@T_2 \rar["\@H"] &
\@T_3
\end{tikzcd}
\end{equation*}
For general background on 2-categories, see \cite[B1.1]{Jeleph}, \cite[I~Ch.~7]{Bcat}.
\end{remark}

However, there remain annoying technicalities in defining a map (technically, a \emph{2-functor}) from this 2-category of theories to groupoids, as in \cref{eq:interp0}.
This is due to non-canonical coding choices.
To define such a 2-functor on objects, we have to pick, for each countable $\Pi_2$ theory $\@T$, some particular quasi-Polish parametrization $\@M -> X$ of it, and then take its isomorphism groupoid $\Iso_X(\@M)$.
Worse yet, to define the 2-functor on \emph{morphisms}, we have to pick, for an interpretation between theories $\@F : \@T_1 -> \@T_2$, some map $f : X_2 -> X_1$ between the respective chosen parametrizing spaces which realizes the operation on models specified by $\@F$, i.e., so that $(\@M_1)_{f(x)} \cong \@F^{(\@M_2)_x}$.

It is possible to make all of these coding choices in some explicit \emph{ad hoc} manner, and this was done in \cite{HMMMcomp} and \cite{Cscc}.
Here, we take the opportunity to illustrate a more abstract approach, via a standard trick in higher-dimensional category theory: instead of defining the 2-functor \cref{eq:interp0} directly, we first ``cover'' its domain with a bigger 2-category that is equivalent, but contains isomorphic copies that include all possible coding choices beforehand.

\begin{definition}
\label{def:interp-param}
A \defn{quasi-Polish parametrized $\Pi_2$ theory} will mean a tuple $(\@L,\@T,X,\@M,p)$, consisting of a countable $\Pi_2$ theory $\@T$ in some countable language $\@L$, together with some quasi-Polish parametrization $X$ of models of $\@T$ via an étale model $p : \@M -> X$ with $\Sigma_1$ saturations.
We will often abbreviate the tuple $(\@L,\@T,X,\@M,p)$ to $(\@T,X,\@M)$.

A \defn{parametrized $\Sigma_1$ interpretation} $(\@F,f) : (\@L_1,\@T_1,X_1,\@M_1,p_1) -> (\@L_2,\@T_2,X_2,\@M_2,p_2)$ consists of a $\Sigma_1$ interpretation $\@F : \@T_1 -> \@T_2$, a continuous map $f : X_2 -> X_1$, and a (specified, but left nameless) isomorphism of étale structures (over $X_2$) $f^*(\@M_1) \cong \@F^{\@M_2}$:
\begin{equation*}
\begin{tikzcd}
\@M_1 \dar["p_1"'] &
f^*(\@M_1) \cong \@F^{\@M_2} \lar \dar[""{coordinate,name=fp2}] &
\@M_2 \dlar["p_2",""{coordinate,name=p2}] \lar[mapsto,"\@F^*"'] \\
X_1 & X_2 \lar["f"']
\end{tikzcd}
\end{equation*}
Regarding $\@M_1, \@M_2$ as ``continuous maps to countable models'' as in \cref{rmk:estr-cts}, the picture becomes
\begin{equation*}
\begin{tikzcd}
X_1 \dar[two heads,"\@M_1"'] & X_2 \dar[two heads,"\@M_2"] \lar["f"'] \\
\{\text{models of $\@T_1$}\} & \{\text{models of $\@T_2$}\} \lar["\@F^*"'] \\[-1em]
\@T_1 \rar["\@F"] & \@T_2
\end{tikzcd}
\end{equation*}
Two consecutive parametrized interpretations may be composed in the obvious manner.
A \defn{$\Sigma_1$-definable isomorphism} between parametrized interpretations $h : (\@F,f) \cong (\@G,g)$ is simply one between the underlying intepretations $h : \@F \cong \@G$.

We thus get a 2-category of quasi-Polish parametrized $\Pi_2$ theories, parametrized $\Sigma_1$ interpretations, and $\Sigma_1$-definable isomorphisms, which admits canonical maps to both sides of \cref{eq:interp0}:%
\begin{equation}
\label{eq:interp2}
\begin{tikzcd}[column sep=0pt]
& \{\text{q-Pol parametrized $\Pi_2$ theories}\}
    \dlar
    \drar \\
\{\text{ctbl $\Pi_2$ theories}\} &&
\{\text{quasi-Polish groupoids}\}
\end{tikzcd}
\end{equation}
The left leg simply forgets about the parametrizations.
The right leg, a contravariant 2-functor, takes a parametrized theory $(\@T,X,\@M)$ to the isomorphism groupoid $\Iso_X(\@M)$; takes a parametrized interpretation $(\@F,f) : (\@T_1,X_1,\@M_1) -> (\@T_2,X_2,\@M_2)$ to the continuous functor
\begin{align*}
\yesnumber
\label{eq:interp2-functor}
(\@F,f)^* : \Iso_{X_2}(\@M_2) &--> \Iso_{X_1}(\@M_1) \\
\paren[\big]{g : (\@M_2)_x \cong (\@M_2)_y} &|--> \smash{\paren[\big]{(\@M_1)_{f(x)} = f^*(\@M_1)_x \cong \@F^{\@M_2}_x \overset{g}{\cong} \@F^{\@M_2}_y \cong f^*(\@M_1)_y = (\@M_1)_{f(y)}};}
\end{align*}
and takes a $\Sigma_1$-definable isomorphism $h : (\@F,f) \cong (\@G,g)$ to the continuous natural isomorphism
\begin{align*}
\yesnumber
\label{eq:interp2-nattrans}
h^* : X_2 &--> \Iso_{X_1}(\@M_1) \\
x &|--> \smash{\paren[\big]{(\@M_1)_{f(x)} = f^*(\@M_1)_x \cong \@F^{\@M_2}_x \overset{h_x}{\cong} \@G^{\@M_2}_x \cong g^*(\@M_1)_x = (\@M_1)_{g(x)}}}
\end{align*}
between the continuous functors $(\@F,f)^*$ and $(\@G,g)^*$ as defined above.
\end{definition}

\begin{theorem}
\label{thm:interp2}
In the above diagram \cref{eq:interp2}:
\begin{enumerate}[label=(\alph*)]
\item \label{thm:interp2:right}
The right leg is a full and faithful contravariant 2-functor, i.e., restricts to an equivalence between each hom-groupoid of parametrized interpretations $(\@T_1,X_1,\@M_1) -> (\@T_2,X_2,\@M_2)$ and the corresponding hom-groupoid of continuous functors $\Iso_{X_2}(\@M_2) -> \Iso_{X_1}(\@M_1)$.
\item \label{thm:interp2:left}
The left leg restricts to an equivalence of 2-categories between the full sub-2-category of \emph{zero-dimensional Polish} parametrized $\Pi_2$ theories (meaning the base space $X$ of the parametrization is zero-dimensional Polish, not necessarily the étale space $M$), and all countable $\Pi_2$ theories.
\end{enumerate}
Thus, we have a composite full and faithful 2-functor, i.e., ``embedding of 2-categories'',
\begin{equation}
\label{eq:interp2-equiv}
\{\text{ctbl $\Pi_2$ theories}\}
\simeq \{\text{0-d Pol parametrized $\Pi_2$ theories}\}
--> \{\text{q-Pol groupoids}\}
\end{equation}
taking each countable $\Pi_2$ theory to the isomorphism groupoid of any of its zero-dimensional Polish parametrizations with $\Sigma_1$ saturations.
\end{theorem}
\begin{proof}
\cref{thm:interp2:right} follows from the Joyal--Tierney theorem.
Indeed, for a parametrized interpretation $(\@F,f) : (\@T_1,X_1,\@M_1) -> (\@T_2,X_2,\@M_2)$,
$\@F$ is a model of $\@T_1$ in the category of $\Sigma_1$ imaginaries over $\@T_2$,
which by Joyal--Tierney in the form of \cref{rmk:jt} is equivalently a model $\@F^{\@M_2}$ in the category of second-countable étale $\Iso_{X_2}(\@M_2)$-spaces,
i.e., an étale model of $\@T_1$ over $X_2$ equipped with a continuous action of $\Iso_{X_2}(\@M_2)$ via isomorphisms;
this action corresponds, via the isomorphism $f^*(\@M_1) \cong \@F^{M_2}$, to the extension of $f : X_2 -> X_1$ to a functor $\Iso_{X_2}(\@M_2) -> \Iso_{X_1}(\@M_1)$ via \cref{eq:interp2-functor}.
It is straightforward to check that this correspondence is functorial via \cref{eq:interp2-nattrans}.

For \cref{thm:interp2:left}, the left leg is essentially surjective, i.e., every countable $\Pi_2$ theory has a zero-dimensional Polish parametrization, by \cref{ex:estr-per}; and it is locally full and faithful, i.e., bijective on 2-cells between each fixed pair of morphisms, by definition of said 2-cells as $\Sigma_1$-definable isomorphisms.
It remains only to check that it is locally (essentially) surjective, i.e., for two zero-dimensional Polish parametrized theories $(\@T_1,X_1,\@M_1), (\@T_2,X_2,\@M_2)$, every $\Sigma_1$ interpretation $\@F : \@T_1 -> \@T_2$ can be parametrized via some continuous map $f : X_2 -> X_1$ and isomorphism $f^*(\@M_1) \cong \@F^{\@M_2}$.
That is, for each fiber $(\@M_2)_x$ of $\@M_2$, we know the model $\@F^{(\@M_2)_x} |= \@T_1$ is isomorphic to some fiber $(\@M_1)_{f(x)}$ of $\@M_1$; we need to find the fiber $f(x)$ and the isomorphism in a continuous manner.
Form the space
\begin{equation*}
\Iso_{X_1,X_2}(\@M_1,\@F^{\@M_2}) = \set[\big]{(y,x,g)}{y \in X_1 \AND x \in X_2 \AND g : (\@M_1)_x \cong \@F^{(\@M_2)_x}}
\end{equation*}
as in \cref{rmk:isogpd2}.
By \cref{rmk:estr-osat-isogpd2}, the second projection $\cod : \Iso_{X_1,X_2}(\@M_1,\@F^{\@M_2}) -> X_2$ is open, and it is also surjective, since as noted before, each $(\@M_2)_x$ is isomorphic to some $(\@M_1)_y$.
Thus by Michael's selection theorem \cite[1.4]{Mselect2} (see also \cite{Mselect0}, \cite{dBPSovert}; to deal with the fact that $\Iso_{X_1,X_2}(\@M_1,\@F^{\@M_2})$ is quasi-Polish instead of Polish, use the latter paper or \cref{it:qpol-openquot}), $\cod$ has a continuous section $X_2 -> \Iso_{X_1,X_2}(\@M_1,\@F^{\@M_2})$, whose first coordinate yields $f$ and third coordinate yields the isomorphism $f^*(\@M_1) \cong \@F^{\@M_2}$.
\end{proof}

\begin{remark}
\label{rmk:interp2}
If one is only interested in certain restricted kinds of theories, then it suffices in \cref{thm:interp2} to restrict to a full sub-2-category of zero-dimensional Polish parametrizations which are known to parametrize all such theories.

For example, if one is only interested in theories with no finite models, and modulo which negated atomic formulas are equivalent to $\Sigma_1$ formulas (e.g., from Morleyizing to recover the traditional definition of $\Sigma_1$ as in \cref{ex:morley-neg}), then it suffices to consider the usual Polish space of models on $\#N$ as in \cref{ex:estr-infdec}, thereby recovering the boldface version of \cite[1.5]{HMMMcomp}.

Likewise, if one wants to keep track of positive atomic formulas other than $=$, then applying \cref{thm:interp2} to the parametrization of \cref{ex:estr-findecs} recovers the boldface version of \cite[3.3]{CMRpos}.
(In fact, as long as $\ne$ is $\Sigma_1$, one may replace ``quasi-Polish groupoids'' with ``zero-dimensional Polish groupoids'' in \cref{eq:interp2-equiv}, by \cref{rmk:estr-findecs-0d,rmk:isogpd-zpol}.)
\end{remark}

\section{The Lopez-Escobar theorem and $\mathcal{L}_{\omega_1\omega}$ imaginaries}

An important tool in descriptive set theory is the Baire category quantifier $\exists^*$ (and its dual $\forall^*$); see \cite[\S8.J, 22.22]{Kcdst}, \cite[\S A]{MTrep}, \cite[\S2.3--4]{Cbk}.
Given a continuous open map $f : X -> Y$ between quasi-Polish spaces, each $A \in X$ has a ``Baire-categorical image'' under $f$:
\begin{equation*}
\exists^*_f(A) := \set[\big]{y \in Y}{A \text{ is nonmeager in the fiber } f^{-1}(y)}.
\end{equation*}
The usefulness of $\exists^*_f$ is largely because it, unlike ordinary image, preserves Borel sets; in fact, it preserves $\*\Sigma^0_\alpha$ sets for all $\alpha$.

We have an analogous result for étale structures $\@M -> X$ with $\Sigma_1$ saturations, thought of as ``maps $x |-> \@M_x : X -> \{\text{all structures}\}$'' as in \cref{rmk:estr-osat-cts}.
The ``fibers'' of such a map should be thought of as the isomorphism orbits $\Iso_X(\@M) \cdot \vec{a}$ of each $\vec{a} \in M^n_X$ (and not the fiber structures $\@M_x$, which are the ``values'' of the map $x |-> \@M_x$).
Thus, ``image'' becomes ``saturation''.

\begin{definition}
\label{def:vaught}
For an étale structure $p : \@M -> X$, open $U \subseteq \Iso_X(\@M)$, and any $A \subseteq M^n_X$, the \defn{Vaught transform} is the ``Baire-categorical saturation''
\begin{equation*}
U * A := \set[\big]{\vec{a} \in M^n_x}{\exists \text{ nonmeagerly many } g \in U \cap \cod^{-1}(x) \text{ s.t.\ } (\vec{a} \in g \cdot A)}.
\end{equation*}
(Here $\cod : \Iso_X(\@M) -> X$ is the codomain map; recall \cref{def:isogpd}.)
\end{definition}

The more common notation for $U * A$ is $A^{\triangle U^{-1}}$, including in \cite{Lgpd} where it was first studied for groupoids; the above notation suggestive of the ordinary saturation $U \cdot A$ is from \cite[4.2.1]{Cbk}.

\begin{theorem}[Lopez-Escobar for étale structures]
\label{thm:lopez-escobar}
Let $\@L$ be a countable language, $p : \@M -> X$ be a second-countable étale $\@L$-structure with $\Sigma_1$ saturations over a quasi-Polish $X$.
For any $\*\Sigma^0_\alpha$ set $A \subseteq M^n_X$, there is a $\Sigma_\alpha$ formula defining $\Iso_X(\@M) * A$.
\end{theorem}

\begin{remark}
\label{rmk:lopez-escobar}
It follows from this statement that more generally, for any basic open $\brbr{U |-> V} \subseteq \Iso_X(\@M)$, where $U, V \subseteq M^m_X$ are open,
\begin{align*}
\brbr{U |-> V} * A
={}& \set[\big]{\vec{a} \in M^n_x}{\{g \in \brbr{U |-> V} \mid \vec{a} \in g \cdot A\} \text{ is nonmeager in } \cod^{-1}(x)} \\
={}& \set[\big]{\vec{a} \in M^n_x}{\exists \vec{b} \in V_x\, (\{g \mid (\vec{a},\vec{b}) \in g \cdot (A \times_X U)\} \text{ is nonmeager in } \cod^{-1}(x))} \\
={}& \set[\big]{\vec{a} \in M^n_x}{\exists \vec{b} \in V_x\, ((\vec{a},\vec{b}) \in \Iso_X(\@M) * (A \times_X U))}
\end{align*}
is the fiberwise inverse image of $V$ under a binary relation $\phi^\@M \subseteq M^n_X \times_X M^m_X$, for some $\Sigma_\alpha$ formula $\phi$ with $n+m$ variables depending only on $U$.
\end{remark}

\begin{proof}[Proof of \cref{thm:lopez-escobar}]
By induction on $\alpha$.
For $\alpha = 1$, this is just the fact that $\@M$ has $\Sigma_1$ saturations.
For a countable union of sets for which the result holds, we may take the disjunction of the formulas.
By \cref{def:borel} of the Borel hierarchy, it thus remains to show, assuming the result holds for all $\*\Sigma^0_\alpha$ sets, that for any two $A, B \in \*\Sigma^0_\alpha(M^n_X)$, the result holds for $A \setminus B \in \*\Sigma^0_{\alpha+1}(M^n_X)$.
Fix countable bases $\@U_m$ for each $M^m_X$, so that $\brbr{U |-> V}$ for $U, V \in \@U_m$ form a basis for $\Iso_X(\@M)$.
For $\vec{a} \in M^n_x$,
\begin{align*}
\vec{a} \in \Iso_X(\@M) * (A \setminus B)
&\iff  \{g \mid \vec{a} \in g(A \setminus B)\} \text{ is nonmeager in } \cod^{-1}(x) \\
\intertext{which by the property of Baire is}
&\iff  \exists m\, \exists U, V \in \@U_m\, \paren*{
    \begin{aligned}
    &\{g \in \brbr{U |-> V} \mid \vec{a} \in gA\} \text{ nonmeager}, \\
    &\{g \in \brbr{U |-> V} \mid \vec{a} \in gB\} \text{ meager}
    \end{aligned}
} \\
&\iff  \exists m\, \exists U, V \in \@U_m\, \paren[\big]{\vec{a} \in (\brbr{U |-> V} * A) \setminus (\brbr{U |-> V} * B)}; \\
\intertext{
by the induction hypothesis and preceding remark, there are $\Sigma_\alpha$ formulas $\phi_U, \psi_U$ such that this is}
&\iff  \exists m\, \exists U, V \in \@U_m\, \paren[\big]{(\exists \vec{b} \in V_x)\, \phi_U^\@M(\vec{a},\vec{b}) \wedge \neg (\exists \vec{c} \in V_x)\, \psi_U^\@M(\vec{a},\vec{c})} \\
&\iff  \exists m\, \exists U \in \@U_m\, \exists \vec{b} \in M^m_x\, \paren[\big]{\phi_U^\@M(\vec{a},\vec{b}) \wedge \neg \psi_U^\@M(\vec{a},\vec{b})}
\end{align*}
where $\Longleftarrow$ is because $\vec{b}$ belongs to some open section $V \in \@U_m$.
So this is defined by the $\Sigma_{\alpha+1}$ formula
\begin{align*}
\bigvee_m \bigvee_{U \in \@U_m} \exists y_0, \dotsc, y_{m-1}\, \paren[\big]{\phi_U(\vec{x},\vec{y}) \wedge \neg \psi_U(\vec{x},\vec{y})}.
&\qedhere
\end{align*}
\end{proof}

\begin{example}
Applying \cref{thm:lopez-escobar} to the standard Polish space of models on $\#N$ (\cref{ex:estr-infdec}) recovers the classical Lopez-Escobar theorem \cite{Llo1o}, or rather its strengthening adapted levelwise to the Borel hierarchy by Vaught \cite{Vaught} (see also \cite[16.8]{Kcdst}, \cite[11.3.6]{Gidst}).

Applying it instead to the parametrization of \cref{ex:estr-findecs} yields the version of the Lopez-Escobar theorem used in \cite{Cscc}, which is the boldface version of the effective ``positive'' (but still admitting $\ne$ as $\Sigma_1$) Lopez-Escobar theorem in \cite{BFRSVle}.
\end{example}

We now have the Borel analogs of the material from the two preceding sections:

\begin{definition}
For a second-countable étale structure $p : \@M -> X$ over quasi-Polish $X$, a \defn{(fiberwise) countable Borel $\Iso_X(\@M)$-space over $X$} is a standard Borel space $A$ equipped with a countable-to-1 Borel map $q : A -> X$ and a Borel action of $\Iso_X(\@M)$.
\end{definition}

\begin{definition}
For a countable $\@L_{\omega_1\omega}$ theory $\@T$, a \defn{$\@L_{\omega_1\omega}$ imaginary sort $\Phi = (\bigsqcup_i \phi_i)/(\bigsqcup_{i,j} \epsilon_{ij})$ over $\@T$} is defined exactly as in \cref{def:imag}, except that the formulas $\phi_i, \epsilon_{ij}$ may be $\@L_{\omega_1\omega}$ instead of $\Sigma_1$.
These may be interpreted in a second-countable étale $p : \@M -> X$ over quasi-Polish $X$ to yield a countable Borel $\Iso_X(\@M)$-space $\Phi^\@M -> X$.
\end{definition}

\begin{theorem}
\label{thm:scc}
Let $\@T$ be a countable $\@L_{\omega_1\omega}$ theory, $p : \@M -> X$ be a second-countable étale space with $\Sigma_1$ saturations parametrizing models of $\@T$.
Then every countable Borel $\Iso_X(\@M)$-space $q : A -> X$ is isomorphic to $\Phi^\@M$ for some $\@L_{\omega_1\omega}$ imaginary $\Phi$ over $\@T$.
\end{theorem}
This was proved in \cite{Cscc} for a particular parametrization $\@M$, namely that of \cref{ex:estr-findecs} (and implicitly in \cite{HMMborel} for the standard parametrization of \cref{ex:estr-infdec}).
\begin{proof}
By \cite[proofs of 4.5.13 and 4.3.9]{Cbk}, we may topologically realize $A$ as a second-countable étale $\Iso_X(\@M)$-space, after refining the topology on the space of objects $X$ of the groupoid $\Iso_X(\@M)$ by adjoining countably many sets of the form $\brbr{U_i |-> V_i} * B_i$ to the topology, where $U_i, V_i \subseteq M^{m_i}_X$ are open and $B_i \subseteq X$ are Borel.
By \cref{rmk:lopez-escobar}, each
$\brbr{U_i |-> V_i} * B_i = p(U_i \cap \phi_i^\@M)$
for some $\@L_{\omega_1\omega}$ formula $\phi_i(x_0,\dotsc,x_{m_i-1})$.
Morleyize the formulas $\phi_i$ to obtain a new theory $\@T'$ in an expanded language $\@L' \supseteq \@L$, and then Morleyize the étale structure $\@M$ via \cref{def:estr-morley}, to obtain a new étale structure $p' : \@M' -> X'$, now with $\Sigma_1$ saturations by \cref{thm:estr-osat-morley}, such that $X'$ is $X$ with a finer topology in which each $p(U_i \cap \phi_i^\@M)$ becomes open.
Pulling back the topologically realized étale space $A$ to $X'$, we thus obtain a second-countable étale $\Iso_{X'}(\@M')$-space, which by the Joyal--Tierney \cref{thm:jt} is named by some $\Sigma_1$ imaginary in $\@M'$, hence by an $\@L_{\omega_1\omega}$ imaginary in $\@M$.
\end{proof}

\begin{remark}
\label{rmk:scc}
As in \cref{rmk:jt}, it follows that we in fact have an equivalence of categories
\begin{align*}
\{\text{$\@L_{\omega_1\omega}$ imaginaries over $\@T$}\} &\overset{\sim}{-->} \{\text{countable Borel $\Iso_X(\@M)$-spaces}\} \\
\Phi &|--> \Phi^\@M
\end{align*}
where the subobjects on the left, namely $\@L_{\omega_1\omega}$-definable subsorts of $\@L_{\omega_1\omega}$ imaginaries, correspond to $\Iso_X(\@M)$-invariant Borel subspaces by the Lopez-Escobar \cref{thm:lopez-escobar}.
\end{remark}

This recovers one of the main results of \cite{Cscc} (extending the boldface result of \cite{HMMborel}).
To recover the rest, we need to convert \cref{thm:scc} into the 2-categorical form of \cref{thm:interp2}.

\begin{definition}
\label{def:interp-borel}
An \defn{$\@L_{\omega_1\omega}$ interpretation} between two $\@L_{\omega_1\omega}$ theories is defined exactly as in \cref{def:interp}, except that the formulas and imaginaries used may be $\@L_{\omega_1\omega}$ rather than $\Sigma_1$.
\end{definition}

\begin{definition}
\label{thm:interp-param-borel}
Recall the notion of \emph{quasi-Polish parametrized $\Pi_2$ theory} $(\@L, \@T, X, \@M, p)$ from \cref{def:interp-param}.
A \defn{quasi-Polish parametrized $\@L_{\omega_1\omega}$ theory} will mean such a tuple where $\@T$ is a countable $\@L_{\omega_1\omega}$ theory in the countable language $\@L$, and $p : \@M -> X$ is a second-countable étale structure over the quasi-Polish space $X$ with $\@L_{\omega_1\omega}$ saturations of open sets.

A \defn{parametrized $\@L_{\omega_1\omega}$ interpretation} $(\@F,f) : (\@L_1,\@T_1,X_1,\@M_1,p_1) -> (\@L_2,\@T_2,X_2,\@M_2,p_2)$ between two such parametrized theories consists of an $\@L_{\omega_1\omega}$ interpretation $\@F : \@T_1 -> \@T_2$, a \emph{Borel} map $f : X_2 -> X_1$, and a \emph{Borel} isomorphism of fiberwise countable \emph{Borel} structures $f^*(\@M_1) \cong \@F^{\@M_2}$.
\end{definition}

We have the Borel analog of the diagram \cref{eq:interp2}:
\begin{equation}
\label{eq:interp2-borel}
\begin{tikzcd}[column sep=0pt]
& \{\text{q-Pol parametrized $\@L_{\omega_1\omega}$ theories}\}
    \dlar["\simeq"']
    \drar[hook] \\
\{\text{ctbl $\@L_{\omega_1\omega}$ theories}\} &&
\{\text{quasi-Polish groupoids}\}
\end{tikzcd}
\end{equation}
in which the three 2-categories and both 2-functors are as before (see \cref{eq:interp2-functor} and \cref{eq:interp2-nattrans}), but with $\@L_{\omega_1\omega}$ interpretations and definable isomorphisms and \emph{Borel} functors and natural transformations.

\begin{theorem}
\label{thm:interp2-borel}
In the above diagram:
\begin{enumerate}[label=(\alph*)]
\item \label{thm:interp2-borel:right}
The right leg is a full and faithful 2-functor.
\item \label{thm:interp2-borel:left}
The left leg is an equivalence of 2-categories (on its entire domain).
\end{enumerate}
Thus, we have a composite full and faithful 2-functor
$\{\text{ctbl $\@L_{\omega_1\omega}$ theories}\}
-> \{\text{quasi-Polish groupoids}\}$.
\end{theorem}
\begin{proof}
As in \cref{thm:interp2}, \cref{thm:interp2-borel:right} follows from \cref{thm:scc} or rather \cref{rmk:scc}.
For \cref{thm:interp2-borel:left}, essential surjectivity follows from the fact that every countable $\@L_{\omega_1\omega}$ theory has a quasi-Polish (or even zero-dimensional Polish) parametrization with $\@L_{\omega_1\omega}$ saturations by Morleyization and any of the examples from \cref{sec:param} that admit finite models, and local full faithfulness is trivial as before.
For fullness, we again follow the proof of \cref{thm:interp2}; the added ingredient needed is that (in the notation from there) surjectivity of the Borel map $\cod : \Iso_{X_1,X_2}(\@M_1, \@F^{\@M_2}) -> X_2$ is already enough to imply the existence of a \emph{Borel} section, which follows by a straigtforward application of Kechris's large section uniformization theorem \cite[18.6$^*$]{Kcdst*} applied to the meager ideal; see \cite[7.9]{Cqpol}.
\end{proof}

\section{Groupoid representations}
\label{sec:gpd}

\Cref{thm:interp2,thm:interp2-borel} show that the operation of taking the semantics of a theory is an ``embedding'' from the 2-category of theories to the 2-category of groupoids (in both the continuous and Borel settings).
It is natural to ask what is the image of this 2-functor, i.e, which quasi-Polish groupoids arise as the groupoid of models of a theory.
The corresponding question in topos theory was answered by Moerdijk~\cite{Mtop2}, who proved what can be called the topos-theoretic analog of the Yoneda lemma (which says that absent any topological structure, \emph{every} groupoid is canonically a groupoid of isomorphisms between structures, via the left translation action on itself).
In this final section, we adapt Moerdijk's result to the countable model-theoretic setting.

First, we clarify what we mean by an abstract topological groupoid (e.g., $\Iso_X(\@M)$):

\begin{definition}
\label{def:gpd-subgpd}
A \defn{topological groupoid} $G \rightrightarrows X$ consists of topological spaces $G, X$ of \emph{morphisms} and \emph{objects} respectively, with continuous \emph{domain} and \emph{codomain} maps $\dom, \cod : G \rightrightarrows X$, as well as \emph{identity} $1_{(-)} : X -> G$, \emph{inverse} $(-)^{-1} : G -> G$, and \emph{composition} $\circ$ (or simply juxtaposition) of adjacent morphisms, subject to the usual axioms of associativity, identity, and inverse.

As is common when working with topological groupoids, we identify each object $x \in X$ with the corresponding identity morphism $1_x \in G$, so that $X \subseteq G$ is a subspace.

By an \defn{open subgroupoid} of $G$, we mean an open subset $U \subseteq G$ of morphisms which is closed under composition and inverse.
It follows that for each $g : x -> y \in U$, the identity morphisms at both its source $1_x = g^{-1} \circ g$ and target $1_y = g \circ g^{-1}$ are in $U$ as well, whence we may regard $U$ as the space of morphisms of a groupoid in its own right, with objects $\{x \in X \mid 1_x \in U\} = \dom(U)$.

We call $G$ a \defn{non-Archimedean} topological groupoid if every identity morphism $1_x \in G$ has a neighborhood basis of open subgroupoids.
\end{definition}

\begin{definition}
\label{def:ogpd-coset}
Now suppose $G \rightrightarrows X$ is an \emph{open} topological groupoid, i.e., $\dom, \cod$ are open maps (see \cref{thm:estr-osat-isogpd}).
For an open subgroupoid $U \subseteq G$ as above, let (by an abuse of notation)
\begin{equation*}
G/U = \dom^{-1}(U)/U = \set[\big]{gU}{g \in G \AND \dom(g) = 1_{\dom(g)} \in U}
\end{equation*}
denote the \defn{space of left cosets} of $U$.
This is naturally an étale $G$-space over $X$ via $\cod : G/U -> X$ and the left multiplication action of $G$.
An open section $S/U \subseteq G/U$ may be identified with its lift in $G$, which is an open right-$U$-invariant subset $S \subseteq \dom^{-1}(U)$.

For two open subgroupoids $U, V \subseteq G$ and an open right-$V$-invariant $S \subseteq \dom^{-1}(V)$, we have
\begin{align}
\label{eq:ogpd-coset-rmul}
U \subseteq SS^{-1}
&\iff  \text{the right multiplication map } (-)S : G/U -> G/V \text{ is well-defined}.
\end{align}
Note also that such a right multiplication map, when defined, is clearly left-$G$-equivariant.
\end{definition}

\begin{definition}
Let $G \rightrightarrows X$ be an open non-Archimedean topological groupoid.
Fix a family $\@U$ of open subgroupoids forming a neighborhood basis for each identity morphism, and for each $U \in \@U$, an open cover $\@S_U$ of $\dom^{-1}(U)$ by open right-$U$-invariant sets (corresponding to a cover of $G/U$ by open sections).
Note that when $G$ is second-countable, $\@U$ and each $\@S_U$ may be chosen to be countable, whence each $G/U$ is a second-countable étale space over $X$.

The \defn{canonical $G$-structure} $\@M = \@M_{G,\@U,(\@S_U)_U}$ (determined by the $\@U, \@S_U$) is the multi-sorted étale structure over $X$ with a sort $G/U$ for each $U \in \@U$, and a unary function $(-)S : G/U -> G/V$ for each $U, V \in \@U$ and $S \in \@S_V$ obeying \cref{eq:ogpd-coset-rmul}.
The canonical left translation action of $G$ on each $G/U$ is an action via isomorphisms between fibers of $\@M_G$, yielding a canonical functor
\begin{align*}
\iota : G &--> \Iso_X(\@M)
\end{align*}
which is the identity on objects and sending each $g : x -> y \in G$ to its action $\@M_x \cong \@M_y$.
\end{definition}

\begin{theorem}[Moerdijk \cite{Mtop2}]
\label{thm:onagpd-dense}
For any open non-Archimedean $T_0$ topological groupoid $G$ with canonical structure $\@M = \@M_{G,\@U,(\@S_U)_U}$ as above, the canonical functor $\iota : G -> \Iso_X(\@M)$ is a topological embedding with $\cod$-fiberwise dense image.
\end{theorem}
\begin{proof}
First, we convert to an alternate description of $\Iso_X(\@M)$.
An isomorphism $f : \@M_x -> \@M_y$ in this groupoid is determined by its values $f(gU)$ for each element $gU$ of each sort $(G/U)_x$ of $\@M_x$.
But preservation of the right multiplication maps $(-)S$ means that $f$ is in fact determined by its values on just the identity cosets $1_x U$: indeed, for any other coset $gU \in (G/U)_x$, we may find
\begin{align}
\label{eq:onagpd-dense-rmul}
g \in S \in \@S_U  \AND  1_x \in V \in \@U  \AND  V \subseteq SS^{-1}
\implies  f(gU) = f(1_x V S) = f(1_x V) S.
\end{align}
(This is the part most analogous to the Yoneda lemma.)
Moreover, these values $f(1_x U)$ must obey
\begin{equation}
\label{eq:onagpd-dense-cohere}
1_x \in U \subseteq V \in \@U  \implies  f(1_x U) V = f(1_x V),
\end{equation}
by finding $1_x \in S \in \@S_U$ and $1_x \in T \in \@S_V$ and then a neighborhood $1_x \in W \subseteq SS^{-1} \cap TT^{-1}$ in $\@U$, so that
$f(1_x U) V = f(1_x W S) V = f(1_x W) SV = f(1_x W) T = f(1_x WT) = f(1_x V)$.
In other words,
\begin{align}
\label{eq:onagpd-dense-invlim}
\{\text{isomorphisms } f : \@M_x \cong \@M_y \in \Iso_X(\@M)\}
&\cong \projlim_{1_x \in U \in \@U} (G/U)_y,
\end{align}
the set of all families $(f(1_x U) \in (G/U)_y)_{1_x \in U \in \@U}$ obeying the coherence condition \cref{eq:onagpd-dense-cohere}; it is easily checked that conversely, any such coherent family yields an isomorphism $f : \@M_x -> \@M_y$ via \cref{eq:onagpd-dense-rmul}.

Moreover, the above bijections \cref{eq:onagpd-dense-invlim} determine the topology of $\Iso_X(\@M)$ in the following sense.
For each $U \in \@U$ and open right-$U$-invariant $S \subseteq \dom^{-1}(U)$, the set
\begin{equation}
\label{eq:onagpd-dense-basis}
\brbr{U/U |-> S/U} = \set[\big]{f : \@M_x \cong \@M_y}{1_x \in U \AND f(1_x U) \subseteq S} \subseteq \Iso_X(\@M),
\end{equation}
corresponding via \cref{eq:onagpd-dense-invlim} to the open section $S/U \subseteq G/U$, is open.
We claim that these sets together with the maps $\dom, \cod : \Iso_X(\@M) \rightrightarrows X$ suffice to generate the topology of $\Iso_X(\@M)$ (from \cref{def:isogpd}).
Indeed, by \cref{eq:onagpd-dense-rmul}, a subbasic open $\brbr{S/U |-> T/U} \subseteq \Iso_X(G)$, where $S/U, T/U \subseteq G/U$ are basic open sections, may be written in terms of the sets \cref{eq:onagpd-dense-basis} as
\begin{align*}
\brbr{S/U |-> T/U} =
\bigcup_{\@U \ni V \subseteq SS^{-1}}
 \bigcup_{S' \subseteq S; T'/V \subseteq G/V; T'S' \subseteq T}
    \brbr{V/V |-> T'/V}
\end{align*}
(where $T'/V \subseteq G/V$ is an open section containing $f(1_x V)$ in \cref{eq:onagpd-dense-rmul}).

Now the $\iota$-preimage of such a set $\brbr{U/U |-> S/U}$ consists of all $g : x -> y \in G$ such that $1_x \in U$ and $gU \in S/U$, which is just the open right-$U$-invariant $S \subseteq G$.
Since $G$ is non-Archimedean, its open right-$U$-invariant sets over all $U \in \@U$ are easily seen to form a basis; thus $\iota$ is an embedding.

Finally, the image of $\iota$ is $\cod$-fiberwise dense: from above, $\Iso_X(\@M)$ has an open basis of sets
\begin{align*}
\bigcap_{i < n} \brbr{U_i/U_i |-> S_i/U_i} \cap \dom^{-1}(V) \cap \cod^{-1}(W)
\end{align*}
where $U_0, \dotsc, U_{n-1} \in \@U$, $S_i/U_i \subseteq G/U_i$ are open sections, and $V, W \subseteq X$ are open.
If such a set is nonempty in some $\cod^{-1}(y)$, then it contains some isomorphism $f : \@M_x -> \@M_y$ where $x = 1_x \in V \cap \bigcap_i U_i$, $y \in W$, and each $f(1_x U_i) \subseteq S_i$.
Let $U \in \@U$ with $1_x \in U \subseteq V \cap \bigcap_i U_i$, and pick $g \in f(1_x U)$.
Then $g : x -> y \in G$ with each $gU_i \supseteq gU \subseteq f(1_x U) \subseteq f(1_x U_i)$, whence $gU_i = f(1_x U_i)$ since both are left cosets of $U_i$, whence $\iota(g) \in \bigcap_{i < n} \brbr{U_i/U_i |-> S_i/U_i} \cap \dom^{-1}(V) \cap \cod^{-1}(y)$.
\end{proof}

\begin{theorem}
\label{thm:onagpd-qpol}
For an open non-Archimedean quasi-Polish groupoid $G$, with canonical second-countable étale structure $\@M = \@M_{G,\@U,(\@S_U)_U}$ with respect to some countable $\@U, \@S_U$ as above, the canonical functor $\iota : G -> \Iso_X(\@M)$ is a topological groupoid isomorphism.
\end{theorem}
\begin{proof}
By the preceding theorem, we may identify $G$ with a $\cod$-fiberwise dense quasi-Polish (hence $\cod$-fiberwise dense $\*\Pi^0_2$) subgroupoid of $\Iso_X(\@M)$; since $(-)^{-1}$ is a homeomorphism, $G \subseteq \Iso_X(\@M)$ is also $\dom$-fiberwise dense $\*\Pi^0_2$.
By Pettis's theorem for the left translation action of the open quasi-Polish groupoid $G$ on the bundle $\cod : \Iso_X(\@M) -> X$ (proved the same way as Pettis's theorem for Polish groups; see \cite[4.2.8]{Cbk}), it follows that $G = GG = \Iso_X(\@M)\Iso_X(\@M) = \Iso_X(\@M)$.
\end{proof}

\begin{remark}
\label{rmk:onagpd-qpol}
We note that the last step above of applying Pettis's theorem also has precedent in topos theory, namely the ``closed subgroupoid theorem'' of Johnstone \cite{Jgpd}; see \cite{Cpettis}.
In the terminology of \cite{Mtop1}, \cite{Mtop2}, we have shown that an open quasi-Polish groupoid is non-Archimedean iff it is \emph{étale-complete}.
\end{remark}

\begin{corollary}
\label{thm:interp2-equiv}
The 2-functor from \cref{thm:interp2} is an equivalence of 2-categories
\begin{equation*}
\{\text{countable $\Pi_2$ theories}\} \simeq \{\text{open non-Archimedean quasi-Polish groupoids}\}
\end{equation*}
(with continuous functors and natural isomorphisms on the right).
\end{corollary}

\begin{corollary}
\label{thm:interp2-borel-equiv}
The 2-functor from \cref{thm:interp2-borel} is an equivalence of 2-categories
\begin{equation*}
\{\text{countable $\@L_{\omega_1\omega}$ theories}\} \simeq \{\text{open non-Archimedean quasi-Polish groupoids}\}
\end{equation*}
(with Borel functors and natural isomorphisms on the right).
\end{corollary}

As before (see \cref{thm:interp2}), these equivalences of 2-categories take a theory to the isomorphism groupoid of any of its parametrizations, hence are not completely canonical (depending on a choice of parametrization).
Also, in the latter result we may replace ``quasi-Polish'' with ``zero-dimensional Polish'', due to the freedom to Morleyize arbitrarily (see \cref{rmk:interp2}).

\begin{remark}
It may seem a bit strange that \cref{thm:interp2-borel,thm:interp2-borel-equiv} are mostly about the Borel setting, yet still mention quasi-Polish spaces and groupoids.
This is an instance of the subtle interactions between topological and Borel structure in the presence of a group(oid) action, as exemplified by results such as Pettis's theorem and the Becker--Kechris theorem; see \cite{Cbk} for an extended discussion on this point.

To our knowledge, it is an open problem to give a purely Borel-theoretic characterization of the standard Borel groupoids equivalent to the isomorphism groupoid of some countable $\@L_{\omega_1\omega}$ theory.
\end{remark}

\medskip
\noindent
Department of Mathematics\\
University of Michigan\\
Ann Arbor, MI, USA\\
\nolinkurl{ruiyuan@umich.edu}

\end{document}